\documentclass[10pt]{amsart}

\usepackage{srcltx}  
\usepackage{amsmath,mathtools}
\usepackage{amsfonts,amssymb,amsthm,amscd,latexsym,euscript, mathrsfs}

\usepackage[all]{xy}
\SelectTips{cm}{10}

\usepackage{hyperref}

\numberwithin{equation}{section}%
\newtheorem {theorem1}{Theorem}[section]
\newtheorem {theorem}[theorem1]{Theorem}

\newtheorem {corollary}[theorem1]{Corollary}
\newtheorem {proposition}[theorem1]{Proposition}
\newtheorem {lemma}[theorem1]{Lemma}
\theoremstyle{definition}
\newtheorem {definition}[theorem1]{Definition}

\newtheorem {cond}[theorem1]{}
\newtheorem {condstarred}{}
\newtheorem {example}[theorem1]{Example}
\theoremstyle{remark}
\newtheorem {remark}[theorem1]{Remark}

\newcommand{\sS}{\ensuremath{\cal S}}
\newcommand{\Sp}{\ensuremath{\textup{Sp}}}

\newcommand{\funSop}{^{\sS^{\op}}}

\newcommand{\op}{{\ensuremath{\textup{op}}}}

\DeclareMathOperator{\Hor}{\ensuremath{\textup{Hor}}}
\DeclareMathOperator{\Ho}{\ensuremath{\textup{Ho}}}

\newcommand {\cofib} {\ensuremath{\hookrightarrow}}

\newcommand {\fibr} {\ensuremath{\twoheadrightarrow}}

\newcommand {\trivcofib} {\ensuremath{\xhookrightarrow \sim}}

\newcommand {\trivfibr} {\ensuremath{\tilde\twoheadrightarrow}}

\newcommand {\we} {\ensuremath{\xrightarrow \sim}}

\newcommand{\vertmap}[2]{\ensuremath{\overset {#1} {\underset #2 \downarrow}}}

\DeclareMathOperator{\hocolim}{\textup{hocolim}}
\DeclareMathOperator{\holim}{\textup{holim}}

\DeclareMathOperator{\MC}{\textup{Cyl}}

\newcommand{\cal}[1]{\ensuremath{\mathcal #1}}

\newcommand{\cat}[1]{\ensuremath{\EuScript #1}}

\newcommand{\colim}{\ensuremath{\mathop{\textup{colim}}}}

\newcommand{\Id}{\ensuremath{\textup{Id}}}
\def\ev{\ensuremath{\textit{ev}}}
\newcommand{\co}{\colon\thinspace}%
\newcommand{\mc}[1]{\ensuremath{\cat{#1}}}%
\renewcommand{\hom}{\ensuremath{{\rm hom}}}%

\title{Duality and small functors}
\author{Georg Biedermann}
\author{Boris Chorny}

\date{\today}

\begin{document}
\begin{abstract}
The homotopy theory of small functors is a useful tool for studying various questions in homotopy theory. In this paper, we develop the homotopy theory of small functors from spectra to spectra, and study its interplay with Spanier-Whitehead duality and enriched representability in the dual category of spectra.

We note that Spanier-Whitehead duality functor $D\colon \Sp\to \Sp^{\op}$ factors through the category of small functors from spectra to spectra, and construct a new model structure on the category of small functors, which is Quillen equivalent to $\Sp^{\op}$. In this new framework for the Spanier-Whitehead duality, $\Sp$ and $\Sp^{\op}$ are full subcategories of the category of small functors and dualization becomes just a fibrant replacement in our new model structure.
\end{abstract}

\maketitle
\tableofcontents

\section{Introduction}
In this paper we give an extension of Spanier-Whitehead duality by producing a Quillen equivalent model for the opposite category of spectra. \\

\paragraph*{\bf Theorem~\ref{Yoneda Quillen equivalence}} \emph{
Let $Y\co \Sp^{\op}\to\Sp^{\Sp}$ be the Yoneda embedding and $Z$ its left adjoint functor.
There is a Quillen equivalence $Z\co\Sp^{\Sp}\rightleftarrows\Sp^{\op}\! :Y$ for a certain model structure on the category $\Sp^{\Sp}$ of small endofunctors of spectra.} 
\vspace{2mm}

As a consequence we prove the following theorem about enriched representability of small covariant functors from spectra to spectra up to weak equivalence. 
\vspace{2mm}

\paragraph*{\bf Theorem~\ref{dual-Brown}} \emph{Let $F\colon \Sp\to \Sp$ be a small functor. Assume that $F$ takes homotopy pullbacks to homotopy pullbacks and also preserves arbitrary products up to homotopy. Then there exists a cofibrant spectrum $Y$ and a natural transformation $F(-)\to R^{Y}(-)$, inducing a weak equivalence $F(X)\we R^{Y}(X)$ for all fibrant $X\in \Sp$.}
\vspace{2mm}

The definitions of representable and small functors are given at the end of the introduction, before the description of the structural organization of the paper.

Let $\Sp$ denote a closed symmetric monoidal model for the stable homotopy category that is locally presentable, with cofibrant unit $S$, and that satisfies the monoid axiom \cite[Def. 2.2]{ss:monoid}. We call the objects spectra.
In Section~\ref{models-for-spectra} we prove that symmetric spectra \cite{HSS} and Lydakis' pointed simplicial functors \cite{Lydakis} with the linear model structure meet the criteria. 

Taking a fibrant representative $\hat S$ for the sphere spectrum $S$, the Spanier-Whitehead dual of a spectrum $A$ is given by the enriched morphism object  
  $$D A=\hom_{\Sp}(A,\hat S)$$
in $\Sp$. We point out that we do not insist on $A$ to be compact.
It coincides with the classical notion of Spanier-Whitehead dual if $A$ is compact and cofibrant. 
This functor $D\colon \Sp^{\op}\to \Sp$ is adjoint to itself, since
\begin{align*}
  \hom_{\Sp}(A,D B)&\cong\hom_{\Sp}(A,\hom_{\Sp}(B,\hat S))\cong\hom_{\Sp}(B,\hom_{\Sp}(A,\hat S)) \\
                 &\cong\hom_{\Sp^{\op}}(D A, B).
\end{align*}
This adjunction factors through the category $\Sp^{\Sp}$ of small functors:
\[
\xymatrix{
\Sp^{\op}  \ar@/_5pt/@{^(-->}[ddr]_{Y} \ar@/^5pt/@{-->}[rr]^D
                    &  						&\Sp.
								     \ar@/_5pt/@{_(->}[ddl]_W
								     \ar@/^5pt/[ll]^D \\
\\
					 & \Sp^{\Sp}
					\ar@/_5pt/@{-->}[uur]_{\ev_{\hat S}}
					\ar@/_6pt/[uul]_{Z}
}
\]
Here $Y$ is the Yoneda embedding. Further, for all $F\in \Sp^\Sp$ we set
  $$Z(F)=\hom(F, \Id)$$
to be spectrum of natural transformation from $F$ to the identity functor of $\Sp$ and  
  $$ \ev_{\hat S}(F)=F(\hat S)$$
the functor which evaluates every $F$ at the chosen fibrant replacement $\hat S$ of the sphere spectrum $S$.  
For all $A\in \Sp$, we set
  $$W(A)=A\wedge R^{\hat S},$$ 
where $R^{\hat S}=\hom_{\Sp}(\hat S,-)$ is the functor represented by $\hat S$. See Section~\ref{adjunction} for more details.
The left adjoint functors are depicted by the solid arrows.

We view theorem~\ref{Yoneda Quillen equivalence} as another approach to the extension of Spanier-Whitehead duality to non-compact spectra as the one proposed by J.~D.~Christensen and D.~C.~Isaksen \cite{Christensen-Isaksen-duality}, where the model for $\Sp^{\op}$ was constructed on the category of pro-spectra. There is an interesting feature that distinguishes our construction: Proposition~\ref{cor:aleph-0-small} states that every object in $\Sp^{\Sp}$ is weakly equivalent to an $\aleph_0$-small representable functor, which is fibrant and cofibrant in our model structure. Since the category of small functors contains full subcategories equivalent to $\Sp$ and $\Sp^\op$, which intersect precisely at the category of compact spectra (see Lemma~\ref{compact-dual}), we obtain a coherent picture of (extended) Spanier-Whitehead duality for non-compact spectra.

Let us move on to Theorem~\ref{dual-Brown}. How does it relate to other representability theorems?
Roughly speaking, in category theory there are two main types of representability theorems: Freyd representability and Brown representability. 
{\it Freyd representability} theorem takes its origin in the foundational book \cite{Freyd-abelian} by P. Freyd on abelian categories and states that limit preserving set valued functors defined on an \emph{arbitrary} complete category and satisfying the solution set condition are representable. It is intimately related to the celebrated adjoint functor theorem.
The first {\it Brown representability} theorem was proven in a seminal article \cite{Brown} by E.H.~Brown on cohomology theories and states that an \emph{arbitrary} semi-exact functor defined on the homotopy category of pointed connected spaces and taking values in the category of pointed sets is representable.

Both theorems have been applied many times and extended to new frameworks. The main difference between the two representability results is that Freyd's theorem imposes the solution set condition on the functor, while not demanding any set theoretical restrictions from the domain category of the functor. On the other side, Brown's theorem uses in a significant way the presence of a set of small generators in the domain category, while not imposing any set theoretical conditions on the functor itself.

Enriched Freyd representability was proven by M.~Kelly, \cite[4.84]{Kelly}. J.~Lurie, \cite[5.5.2.7]{Lurie} proved the analog in the framework of ($\infty$,1)-categories.
The solution set condition is replaced by the accessibility condition on the functor in both cases. Note that a covariant functor with an accessible category in the domain is small if and only if it is accessible, but the concept of small functor is applicable even if the domain category is not accessible.

The enriched version of Brown representability theorem for contravariant functors from spaces to spaces was proven by the second author in \cite{Chorny-Brownrep}. J.~F.~Jardine \cite{Jardine-Rep} generalized the theorem for functors defined on a cofibrantly generated simplicial model category with a set of compact generators.

The smallness assumption on the functor classifies our theorem as a Freyd-type result up to homotopy. On the other hand, our exactness assumptions on the functor are less restrictive than in Freyd's theorem and closer to a Brown-type theorem. Brown representability for covariant functors from the homotopy category of spectra to abelian groups was proven by A.~Neeman \cite{Neeman-dual}. An enriched version of Neeman's theorem is still not proven.
 
In homotopy theory, there is a third kind of theorem: G.W.~Whitehead's \cite{GWWhitehead} representability of homological functors  where, for a covariant homological functor $F$, an object $C$ is constructed together with an objectwise weak equivalence
  $$ F(-)\stackrel{\simeq}{\longrightarrow} C\wedge (-). $$
Its enriched counterpart was proven by T.~Goodwillie \cite{Goo:calc3} as classification of linear functors. Whithead's representability is related to Brown's representability on finite spectra through the Spanier-Whitehead duality, as it was explained by J.F.~Adams \cite{Adams-Brownrep}. An enriched version of this connection is contained in Lemma~\ref{compact-dual} and is central in our proof of the representability theorem.  

If a category \cat K is enriched in a closed symmetric monoidal category \cat V, then a functor $F\colon \cat K\to \cat V$ is {\it called $($\cat V-enriched$)$ representable} if there exists an object $K\in \cat K$ and a natural isomorphism of functors $\eta\colon F(-)\to \hom_{\cat K}(-,K)$, where $\hom_{\cat K}(-,-)\colon \cat K^{\op}\times \cat K \to \cat V$ is the enriched  $\hom$ functor. Our notation for representable functors is $R_K(-)=\hom_{\cat K}(-,K)$ and $R^K=\hom_{\cat K}(K,-)$. 

A \emph{small functor} from one large category to another is a left Kan extension of a functor defined on a small, not necessarily fixed, subcategory of the domain. Equivalently, if the domain category is enriched over the range category, small functors are small weighted colimits of representable functors. The category of small functors is a reasonable substitute for the non-locally small category of all functors, provided that we are interested in studying global phenomena and not satisfied with changing the universe as an alternative solution. Several variations of this concept for set-valued functors were extensively studied by P.~Freyd  \cite{Freyd}. In algebraic geometry, small functors were used by W.~C.~Waterhouse \cite{Waterhouse} under the name `basically bounded presheaves' in order to treat categories of presheaves over large sites without changing the universe, since such a change might also alter the sets of solutions of certain Diophantine equations. For enriched settings, our main reference is the work of B.~Day and S.~Lack \cite{Lack}. Recently, several applications of small functors from spaces to spaces have appeared in homotopy theory \cite{BCR}, \cite{Chorny-Dwyer}.

The paper is organized as follows.
Section~\ref{fibrant-projective-section} is devoted to the construction of a new model category structure on small functors, which is close to the projective model category, except that weak equivalences and fibrations are determined only on the values of the functors on fibrant objects. Hence, it is called the \emph{fibrant-projective} model structure. Its goal is to create an initial framework in which the adjunction $(Z,Y)$ is a Quillen pair. In Section~\ref{models-for-spectra} we provide model categories for spectra that satisfy the conditions given in the previous section. 
In Section~\ref{section-homotopy} we obtain an auxiliary result~\ref{approx}.
To obtain the promised new Quillen equivalent model for $\Sp^{\op}$, where every spectrum corresponds to a representable functor, we perform in Section~\ref{Q-loc} a non-functorial version of Bousfield-Friedlander's $Q$-construction on $\Sp^{\Sp}$.
This is the crucial technical part of this paper. We localize the fibrant-projective model structure on $\Sp^{\Sp}$ with respect to the ``derived unit" of the adjunction $(Z,Y)$. Our localization construction fails to be functorial; nevertheless, it preserves enough good properties to allow us to get a left Bousfield localization of $\Sp^{\Sp}$ along the lines of the Bousfield-Friedlander localization theorem \cite{BF:gamma}. In the Appendix~\ref{gen-BF}, we provide an appropriate generalization of the Bousfield-Friedlander machinery to encompass non-functorial homotopy localizations.
The represntability theorem \ref{dual-Brown} is derived in the last Section~\ref{sec:rep}.

\subsection*{Acknowledgements}
We would like to that A.~K.~Bousfield for helping us to prove the ``only if" part in the classification of $Q$-fibrations in Theorem \ref{main} and the anonymous referee for many useful suggestions.

\section{Yoneda embedding for large categories}\label{adjunction}

In this article, we consider enriched categories and enriched functors.
Some sources do not distinguish between the cases of small and large domain categories, although functors from large categories have large \hom-sets, i.e., proper classes. Morphism sets and internal mapping objects only make sense after a change of universes.
Unfortunately, we cannot adopt this approach, as the internal mapping objects will play a crucial role in the construction of homotopy theories on functors. Thus, we will use small functors and the Yoneda embedding with values in the category of small functors.

The language of enriched category theory is used throughout the paper. The basic definitions and notations may be found in Max Kelly's book, \cite{Kelly}.

\begin{definition}
Let \cat V be a symmetric monoidal category and \cat K a \cat V-category. A \cat V-functor from \cat K to \cat V is called a {\it small functor} if it is a \cat V-left Kan extension of a \cat V-functor defined on a small but not necessarily fixed subcategory of \cat K. The category of small functors is denoted by $\cat V^{\cat K}$.
\end{definition}

The main example of the symmetric monoidal model category \cat V considered in this paper is the category $\Sp$ of spectra. As explained in Section~\ref{models-for-spectra} we can work with either symmetric spectra, \cite{HSS}, or Lydakis' category of linear functors, \cite{Lydakis}. In the future we hope to extend the ideas of this paper to make them applicable for functors enriched in simplicial sets $\cal S$ or chain complexes, so we record the basic results in bigger generality, than required for the present paper.

\begin{definition}
The {\it enriched covariant Yoneda embedding functor}
   $$ Y\colon \cat K^\op \to\cat V^{\cat K}$$
is given by mapping an object $K$ in $\cat K$ to the {\it \cat V-enriched covariant representable functor}
   $$ R^K\colon\cat K\to\cat V, L\mapsto \hom_{\cat K}(K,L)=R^K(L). $$
\end{definition}

\begin{remark}
For all $K$ the functor $R^K$ is small as it is Kan extended from the full subcategory of \cat K given by the object $K$.
\end{remark}

\begin{definition}
We denote the \cat V-left adjoint to $Y$ to be the end construction $Z(F)= \int_{K\in \cat K} \hom_{\cat V}(F(K), K)$.
\end{definition}
Note that if $\cat K=\cat V$, as we will assume from some point in this paper, then the end in the definition above becomes just a mapping object in the category of small functors $\cat V^{\cat V}$:
\[
\forall F\in \cat V^{\cat V},\quad\, Z(F)= \hom_{\cat V^{\cat V}}(F, \Id_{\cat V}).
\]

We obtain the Yoneda adjunction
\begin{equation}\label{adj-V-K}
Z\co \cat V^{\cat K}\rightleftarrows\cat K^{\op} : \! Y, 
\end{equation}
which we turn into a Quillen adjunction in Proposition \ref{Quillen-adjunction}.

Let us briefly verify that the functor $Z$ is indeed the left adjoint of $Y$. Let $F\in \cat V^{\cat K}$ and $X\in \cat K$, then
\begin{align*}
\hom_{\cat K^{\op}}(Z(F), X) &= \hom_{\cat K}(X, Z(F)) \quad\text{by definition of }Z(F)\\
			&= \hom_{\cat K}(X, \int_{K\in \cat K}K^{F(K)}) \quad\text{by the universal property of an end}\\
			&= \int_{K\in \cat K} \hom_{\cat K}(X, K^{F(K)}) \quad\text{since \cat K is cotensored over \cat V }\\
			&= \int_{K\in \cat K} \hom_{\cat V}(F(K), \hom_{\cat K}(X,K)) \quad					\text{by definition of the object}\\
			&\hspace*{6cm}\text{of natural transformations}\\
			&= \hom_{\cat V^{\cat K}}(F, Y(X)).
\end{align*}

In \cite{Chorny-Brownrep}, $\cat K = \sS^\op$ and the Yoneda embedding $Y\colon \sS\cofib \sS\funSop$ was of central importance. In the current article, we take $\cat K=\Sp$. The Yoneda embedding $Y\colon \Sp^{\op}\cofib \Sp^{\Sp}$ plays an analogous role as before and we will turn the adjunction $(Z,Y)$ into a Quillen equivalence in Theorem \ref{Yoneda Quillen equivalence}.

\section{Homotopy theory of small functors}\label{fibrant-projective-section}

We want the Yoneda adjunction (\ref{adj-V-K}) in the case $\cat V=\Sp$ to be a Quillen pair between suitable model structures on each side. The projective model structure constructed by Chorny and Dwyer \cite{Chorny-Dwyer} on the category of small functors, where weak equivalences and fibrations are objectwise, is not suitable: if we apply $Y(v)=\cat V(-,v)$ to a trivial fibration in $\cat V^{\op}$, aka. a trivial cofibration in $\cat V$, then for non-fibrant $v$ this map will not remain a weak equivalence. So $Y$ is not right Quillen.

We remedy this shortcoming with the following new model structure, which is introduced after we recall a few standard definitions.

\begin{definition}
Let $I$ be a class of maps in a category \cat C. Following standard conventions \cite[10.5.2]{Hirschhorn}, we denote by $I$-inj the class of maps that have the right lifting property with respect to all maps in $I$. We denote by $I$-cof the class of maps that have the left lifting property with respect to all maps in $I$-inj. We denote by $I$-cell the class of relative cell complexes obtained from all maps in $I$ as defined in \cite[10.5.8]{Hirschhorn}.
\end{definition}

\begin{definition}
Let \cat V be a closed symmetric monoidal model category and let $\cat K$ be a $\cat V$-model category. A $\cat V$-natural transformation $f\colon F\to G$ in the category of small functors $\cat V^{\cat K}$ is a \emph{fibrant-projective weak equivalence} (resp., a \emph{fibrant-projective fibration}) if for all fibrant $K\in \cat K$ the map $f(K) \colon F(K)\to G(K)$ of objects of \cat V is a weak equivalence (resp., a fibration). We often abbreviate the functor category $\cat V^{\cat K}$ by $\cat F$.
\end{definition}

The main result of this section is Theorem~\ref{fibrant-projective}, where we show that the fibrant-projective weak equivalences and the fibrant-projective fibrations equip $\cat F$ with a model structure, which is, naturally, called fibrant-projective. 

\begin{definition} We recall the following definitions.
\begin{enumerate}
   \item
A category is \emph{class $\mu$-locally presentable}, \cite{Chorny-Rosicky-1}, if it is complete and cocomplete and has a class $\cal A$ of $\mu$-presentable objects such that every other object is a filtered colimit of the elements of $\cal A$.
It is \emph{class locally presentable} if there is a $\mu$ for which $\cal A$ is \emph{class $\mu$-locally presentable}.
   \item
A model category is \emph{class $\mu$-cofibrantly generated}, \cite{Chorny-Rosicky-2}, if there exist classes of generating (trivial) cofibrations with $\mu$-presentable domains and codomains satisfying the generalized small object argument \cite{pro-spaces}. A model category is \emph{class cofibrantly generated} if it is \emph{class $\mu$-cofibrantly generated} for some cardinal $\mu$. 
   \item
A model category is \emph{class $\mu$-combinatorial} if it is class $\mu$-locally presentable and class $\mu$-cofibrantly generated. A model category is \emph{class combinatorial} if it is class $\mu$-combinatorial for some cardinal $\mu$.
	\item
A $\cat V$-model category is class combinatorial, \cite{Chorny-Rosicky-2}, if its underlying category is so. An object of a $\cat V$-category is $\lambda$-presentable if it is $\lambda$-presentable in the underlying category.
\end{enumerate}
\end{definition}

If in the previous definition the various classes required to exist are in fact sets one recovers the well-known concepts of $\mu$-local presentability, cofibrant generation and so forth.

\begin{definition}[\cite{ss:monoid}]\label{monoid-axiom}
Let ${\rm tcof}_{\mc{V}}$ be the class of trivial cofibrations in  $\mc{V}$. Let $\mathcal{E}_{\mc{V}}$ be the class of relative cell complexes in \mc{V} generated by the class of morphisms
 \[ \{j\otimes A\,|\,j \in{\rm tcof}_{\mc{V}},A\in{\rm ob}\mc{V}\}. \]
The model structure on \mc{V} satisfies the {\it monoid axiom} if every morphism in $\mathcal{E}_{\mc{V}}$ is a weak equivalence.
\end{definition}

\begin{definition}[\cite{DRO:enriched} Def. 4.6]\label{Def. strongly left proper}
Let ${\rm cof}_{\mc{V}}$ be the class of cofibrations in  $\mc{V}$. Let $\mathcal{D}_{\mc{V}}$ be the class of relative cell complexes generated by the class of morphisms
 \[ \{i\otimes A\,|\,i \in{\rm cof}_{\mc{V}},A\in{\rm ob}\mc{V}\}. \]
The model structure on \mc{V} is {\it strongly left proper} if the cobase change of a weak equivalence along any map in $\mathcal{D}_{\mc{V}}$ is a weak equivalence.
\end{definition}

Now we state the main result of this section.
\begin{theorem}\label{fibrant-projective}
Let $\lambda$ be a regular cardinal.
Let $\cat V$ be a closed symmetric monoidal category equipped with a $\lambda$-combinatorial model structure such that the unit $S\in \cat V$ is a cofibrant object and the monoid axiom $\ref{monoid-axiom}$ is satisfied. Let \cat K be a $\lambda$-combinatorial \cat V-model category. Then the category of small functors $\cat V^{\cat K}=\cat F$ with the fibrant-projective weak equivalences, fibrant-projective fibrations and the cofibrations given by the left lifting property is a class-combinatorial \cat V-model category. It is right proper if the model structure on \cat V is. It is left proper if the model structure on \cat V is strongly left proper. 
\end{theorem}

\begin{proof}
The category \cat F is complete by the main result of \cite{Lack} and cocomplete by \cite[Prop.~5.34]{Kelly}.

We use the usual recognition principle \cite[11.3.1]{Hirschhorn} due to Kan to establish the remaining axioms for a class-cofibrantly generated model structure.
\begin{enumerate}
   \item
Weak equivalences are obviously closed under retracts and 2-out-of-3.
   \item
There are classes of generating cofibrations $I_{\mc{F}}$ and trivial cofibrations $J_{\mc{F}}$ defined in \ref{generators-for-fib-proj} that admit the generalized small object argument in the sense of \cite{pro-spaces} as proved in \ref{admit-gen-small-ob-arg}.
   \item
A map is $I_{\mc{F}}$-injective if and only if it is $J_{\mc{F}}$-injective and a weak equivalence by Lemma \ref{right-lifting-properties}.
   \item
Every $J_{\mc{F}}$-cofibration is a weak equivalence by \ref{I-cell-in-we}.
\end{enumerate}
The model structure is a \cat V-model structure by Proposition~\ref{V-model-structure}. Right properness can be checked by evaluating on all fibrant objects in \cat K and then follows from the right properness of \cat V.
Left properness is proved in \cite[4.7,4.8]{DRO:enriched}. The key observation is that any fibrant-projective cofibration is objectwise a retract of maps in $\mc{D}_{\mc{V}}$. 
\end{proof}

\begin{corollary}\label{Quillen-adjunction}
If we equip the category $\cat V^{\cat K}=\cat F$ with the fibrant-projective model structure constructed in Theorem~\emph{\ref{fibrant-projective}}, then the adjunction $(\ref{adj-V-K})$ becomes a Quillen pair.
\end{corollary}

\begin{proof}
In the opposite category $\cat K^{\rm op}$ consider a (trivial) fibration $f^{\rm op}$, which in fact is a (trivial) cofibration $f\colon A\to B$ in \cat K. The induced map $Y(f)\colon R^{B}\to R^{A}$ is a (trivial) fibration in the fibrant-projective model structure, since $\hom(f,W)$ is a (trivial) fibration for every fibrant object $W$ in \cat V. Thus, the functor $Y$ is right Quillen.
\end{proof}

The rest of this section is devoted to the missing steps in the proof of Theorem~\ref{fibrant-projective}. We assume that the closed symmetric monoidal model category \cat V and the \cat V-model category \cat K satisfy the conditions of Theorem~\ref{fibrant-projective}.

The category $\cat V^{\cat K}$ is tensored over $\cat V$ by applying the tensor product of $\cat V$ objectwise:
  $$ (F\otimes V)(K)=F(K)\otimes V $$
for a functor $F$ in $\cat V^{\cat K}$ and objects $V$ in $\cat V$ and $K$ in $\cat K$.

\begin{definition}\label{generators-for-fib-proj}
Let $I_{\cat V}$ and $J_{\cat V}$ be sets of generating cofibrations and generating trivial cofibrations for $\cat V$. We define the following two classes of morphisms in $\cat F$:
\begin{align*}
I_{\mc{F}} &= \left\{R^{X} \otimes A\hookrightarrow R^{X} \otimes B \,\left |\, \vertmap A B \in I_{\cat V}; \; X\in \cat K^{f} \right.\right\},\\
J_{\mc{F}} &=\left\{R^{X} \otimes C \trivcofib R^{X} \otimes D \,\left |\, \vertmap C D \in J_{\cat V}; \; X\in \cat K^{f} \right.\right\},
\end{align*}
where $\cat K^{f}\subset \cat K$ is the subcategory of fibrant objects.
\end{definition}

\begin{remark}
By \cite[Prop.~2.3.3]{Dugger-presentation}, for any $\lambda$-combinatorial model category $\cat K$ there exists a sufficiently large cardinal $\mu$, such that $\cat K$ is $\mu$-combinatorial and there exists a $\mu$-accessible fibrant replacement functor $\widehat{\phantom{X}} \co\mc{K}\to\mc{K}$ sending $\mu$-presentable objects to $\mu$-presentable objects, i.e., for each $X\in \cat K$ there is a natural trivial cofibration $X\trivcofib \hat X$ such that $\hat X$ is $\mu$-presentable whenever $X$ is and $\widehat{\phantom X}$ commutes with $\mu$-filtered colimits.
\end{remark}

From now on and for the rest of the whole article we fix a choice of a fibrant replacement functor $\widehat{\phantom{x}}$ as in the previous remark on the source category.

\begin{remark}
Every fibrant object $X$ is a retract of $\hat X$. It follows that the generating classes $I_{\mc{F}}$ and $J_{\mc{F}}$ can be replaced by the classes
\begin{align*}
I_{\mc{F}}' &= \left\{R^{\hat X} \otimes A\hookrightarrow R^{\hat X} \otimes B \,\left |\, \vertmap A B \in I_{\cat V}; \; X\in \cat K \right.\right\},\\
J_{\mc{F}}' &=\left\{R^{\hat X} \otimes C \trivcofib R^{\hat X} \otimes D \,\left |\, \vertmap C D \in J_{\cat V}; \; X\in \cat K \right.\right\},
\end{align*}
because the retract argument allows one to see that the classes of maps with the respective right lifting properties coincide.
\end{remark}

\begin{definition}
We define $\cal R$ to be the class of maps in $\cat V^\cat K$ that are trivial fibrations when evaluated on all fibrant objects.
We define $\cal T$ to be the class of maps that are fibrations when evaluated on all fibrant objects.
\end{definition}

\begin{lemma} \label{right-lifting-properties}
We have:
\begin{enumerate}
   \item A map is in $\cal R$ if and only if it has the right lifting property with respect to all maps in $I_{\cat F}$: $\cal R=I_{\cat F}$-inj.
   \item A map is in $\cal T$ if and only if it has the right lifting property with respect to all maps in $J_{\cat F}$: $\cal T=J_{\cat F}$-inj.
\end{enumerate}
\end{lemma}

\begin{proof}
Straightforward.
\end{proof}

\begin{lemma}\label{cofibrant representables}
For any fibrant object $X$ in $\cat K$, the canonical map $\emptyset\to R^X$ has the left lifting property with respect to all maps in $\cal R$, i.e., it is in $I_{\mc{F}}$\textup{-cof}.
\end{lemma}

\begin{proof}
Because the unit $S$ of \cat V is cofibrant, it is easy to see that the map $\emptyset\to R^{ X}\otimes S=R^X$ has the left lifting property with respect to all maps in $\cal R$. 
\end{proof}

\begin{lemma}\label{I-cell-in-we}
Every relative $J_{\cat F}$-cell complex is a fibrant-projective weak equivalence.
\end{lemma}

\begin{proof} Because fibrant-projective weak equivalences can be detected by evaluating on fibrant objects, one easily verifies that the lemma follows from the monoid axiom that holds in \cat V.
\end{proof}

Now we present the crucial technical part in the proof of the existence of the fibrant-projective model structure. The generalized small-object argument, \cite{pro-spaces}, may be applied on a class of maps $\cal I$ satisfying certain co-solution set condition (see below), so that on each step of the transfinite induction we could attach one cofibration, through which all other maps in $\cal I$ factor.

\begin{lemma}\label{admit-gen-small-ob-arg}
The classes $I_{\cat F}$ and $J_{\cat F}$ admit the generalized small object argument.
\end{lemma}

\begin{proof}
Since the domains and codomains of the maps in $I_{\cat V}$ and $J_{\cat V}$ are $\lambda$-presentable, so are the maps in $I_{\mc{F}}$ and $J_{\mc{F}}$. It remains to show that $I_{\mc{F}}$ and $J_{\mc{F}}$ satisfy the following co-solution set condition:

(CSSC): Every map $f\colon F\to G$ in $\mc{F}$ may be equipped with a commutative square
\[
\xymatrix{
C \ar[r]\ar@{^{(}->}[d]_{g}& F\ar[d]^{f}\\
D \ar[r]& G,
}
\]
so that $g\in I_{\mc{F}}\text{-cof}$ (resp. $g\in J_{\mc{F}}\text{-cof}$) and every morphism of maps $i\to f$ with $i\in I_{\cat F}$ (resp. $i\in J_{\cat F}$) factors through $g$.

We will prove this condition in the first case, where we construct $g\in I_{\mc{F}}\text{-cof}$. The second case with $g\in J_{\mc{F}}\text{-cof}$ will be dealt with in brackets along the way.
For the proof of (CSSC) we consider a morphism of maps $i\to f$ for some $i\in I_{\cat F}$ (resp. $i\in J_{\cat F}$)  and arbitrary $f$ in \cat F as above. Let the diagram
\begin{equation}\label{i to f}
\xymatrix{
R^{X}\otimes A \ar[r]\ar@{^{(}->}[d]_{i}& F\ar[d]^{f}\\
R^{X}\otimes B \ar[r]& G
}
\end{equation}
be this morphism. Here $A\to B$ is in $I_{\cat V}$ (resp. $J_{\cat V}$) and $X$ is a fibrant object in $\cat K$.
By adjunction, this square corresponds to the following commutative diagram of solid arrows:
\begin{equation}\label{pullback W}
\xymatrix@=20pt{
R^{X} \ar@/^1pc/[drr]\ar@/_1pc/[ddr] \ar@{..>}[rd]|\varphi\\
					&	W\ar@{-->}[r]\ar@{-->}[d]	& F^{A}\ar[d]^{f^{A}}\\
					&G^ B \ar[r] & G^{A}.
}
\end{equation}
where $W=F^{A}\times_{G^{A}}G^{B}$ is the pullback and $\varphi$ is the universal map.
We claim that for such $W$ in (\ref{pullback W}) there exists a map $p\colon\widetilde{W}\tilde \fibr W$ where $p\in I_{\mc{F}}\text{-inj}$, and the canonical map $\emptyset\to\widetilde{W}$ is in $I_{\mc{F}}\text{-cof}$. In other words, $\widetilde{W}$ is a cofibrant replacement of $W$ in the yet to be constructed fibrant-projective model structure.
The proof of this claim will be postponed to Lemma \ref{postponed}.
We proceed with the proof that property (CSSC) holds.

The map $\varphi$ lifts along the map $\widetilde{W}\to W$ by Lemma \ref{cofibrant representables}.
Unrolling the adjunction, we find that the morphism $i\to f$ from (\ref{i to f}) factors through the map
  $$w_{A\to B}\colon\widetilde{W} \otimes A \cofib \widetilde{W} \otimes B,$$
which is in $I_{\cat F }$-cof (resp. $J_{\cat F}$-cof) as we are now going to prove.
We choose the required map $g\colon C\cofib D$ to be
\[
g=\coprod_{\vertmap A B \in I_{\cat V}} w_{A\to B} \,\,\,\,\text{and}\,\,\,\,\quad \left(\text{resp.}\,\coprod_{\vertmap A B \in J_{\cat V}} w_{A\to B} \right)
\]
We need finally to show that $g\in I_{\mc{F}}\text{-cof}$ (resp. $g\in J_{\mc{F}}\text{-cof}$). It suffices to show
  $$w_{A\to B}\in I_{\mc{F}}\text{-cof}\,\,\,\, ({\rm resp.}\, w_{A\to B}\in J_{\mc{F}}\text{-cof})$$
for each $w_{A\to B}\co\widetilde{W}\otimes A\to\widetilde{W}\otimes B$ from above.
So, let $q\colon M\to N$ be an arbitrary map in $\cal S=I_{\cat F}{\rm -inj}$ (resp. $\cal T=J_{\cat F}{\rm -inj}$). Consider any commutative square as follows:
\begin{equation}\label{lift-between-w-and-q}
\xymatrix{
\widetilde{W} \otimes A\ar@{^{(}->}[d]_{w_{A\to B}} \ar[r] & M\ar@{->>}[d]^{q}\\
\widetilde{W} \otimes B\ar[r] \ar@{..>}[ur]& N
}.
\end{equation}
We claim this diagram admits a dotted lift. We actually construct a dotted arrow in the following adjoint solid arrow diagram
\[
\xymatrix@=14pt{
\widetilde{W} \ar@/_1pc/[dddr] \ar@/^1pc/[rrrd] \ar@{..>}[dr]\\
                   & M^{B}\ar[dd]\ar[rr] \ar@{->>}[dr]^{\dir{~}}   &   & M^{A}\ar@{=}[d]\\
                   &                              & P\ar[r]\ar[d]& M^{A}\ar[d]
\\
                   & N^{B}\ar@{=}[r]      & N^{B}\ar[r]        & N^{A},
}
\]
where $P=M^{A}\times_{N^{A}}N^{B}$ denotes the pullback. The induced map $M^{B}\trivfibr P$ is in $\cal S$, which can be checked by evaluating on fibrant objects of $\cal K$ because the model structure on $\cat V$ is monoidal and we are in one of the following cases:
\begin{enumerate}
   \item
The cofibration $A\cofib B$ is a weak equivalence in \cat V. This is the case inside the brackets above;
   \item
The map $q$ is in $\cal S=I_{\cat F}$-inj and hence a trivial fibration when evaluated on fibrant objects. This is the case outside the brackets above.
\end{enumerate}
The dotted arrow exists because we get a lift to $P$ by its universal property and then a lift to $M^B$ since $\emptyset\to\widetilde{W}$ is in $I_{\mc{F}}\text{-cof}$. This corresponds to the lift in the original square (\ref{lift-between-w-and-q}) finishing the proof of property (CSSC).
\end{proof}

\begin{definition}\label{Kmu}
The full subcategory of $\cat K$ given by the $\mu$-presentable objects will be denoted by $\cat K_{\mu}$. 
\end{definition}

In the previous proof, we have used the following
\begin{lemma}\label{postponed}
For each $W$ in diagram $(\ref{pullback W})$ the canonical map $\emptyset\to W$ can be factored into a map $\emptyset\to\widetilde{W}$ in $I_{\mc{F}}$\textup{-cof} followed by $\widetilde{W}\to W$ in $\cal S=I_{\mc{F}}$\textup{-inj}.
\end{lemma}

\begin{proof}
By assumptions, $\cat V$ and $\cat K$ are $\lambda$-combinatorial model categories. We know, by \cite[Prop.~2.3(iii)]{Dugger-presentation}, that there exists a $\lambda$-accessible fibrant replacement functor in $\cat K$ denoted by $\hat{-}$, such that for every sufficiently large regular cardinal $\mu \unrhd \lambda$ and for every $\mu$-presentable object $X$, $\hat X$  is also $\mu$-presentable. We fix this cardinal $\mu$. Here we have chosen $\mu \unrhd \lambda$ so, that every $\lambda$-accessible category is also $\mu$-accessible, \cite{Makkai-Pare}.

The functor $W$ is small. In other words, it is a left Kan extension of a functor defined on a small subcategory $\cat K_W$ of $\cat K$. Alternatively,  we can write $W$ as a weighted colimit of a diagram of representable functors $\cal R\colon \cat K_W^{\op} \to \cat V^{\cat K}$, $\cat K_W \ni K\mapsto R^K$ with a functor of weights $\cal M\colon \cat K_W\to \cat V$. The full image of $\cal M$ is an essentially small subcategory of $\cat V$ denoted by $\cat V_W$. Therefore, $W$ is a colimit of a \emph{set} of functors
   $$\{R^{K}\otimes M\,|\, K\in \cat K_{W},\, M\in\cat V_{W}\}.$$

Enlarging $\mu$ if necessary, we ensure that every $K\in\cat K_{W}$ is $\mu$-presentable. Then every $R^{K}$ is a $\mu$-accessible functor. Hence, $W$ is a $\mu$-accessible functor as a colimit of $\mu$-accessible functors. In order to construct the stated factorization, we apply to the map $\emptyset\to W$ the ordinary small object argument on the following set of maps:
\[
I_{W}=
\left\{R^{\hat X}\otimes A\hookrightarrow  R^{\hat X}\otimes B \,\left|\, \vertmap A B \in I_{\cat V},\, X\in{\cat K}_{\mu}
\right.\right\},
\]
where $\cat K_{\mu}$ is, as in Definition~\ref{Kmu}, the full subcategory of $\cat K$ given by the $\mu$-presentable objects. 
Note, that by the choice of $\mu$, $\hat{X}\in \cat K_{\mu}$ for all $X\in \cat K_{\mu}$.

The map $\emptyset\to\widetilde{W}$ is then in $I_{W}{\text{-cell}}\subset I_{\cat F}$-cof. The natural transformation of functors $p\colon\widetilde{W}\to W$ has the property that for all $X\in \cat K_{\mu}$ the map 
  $$p(\hat X)\colon \widetilde{W}(\hat X)\tilde \fibr W(\hat X)$$ 
is a trivial fibration in \cat V. We need to show that $p\in I_{\cat F}$-inj, i.e. that it is a trivial fibration on all fibrant $X$.

Since $\hat{X}\in \cat K_\mu$ for all $X\in \cat K_{\mu}$, the functors $R^{\hat X}$ are $\mu$-accessible for all $X\in \cat K_{\mu}$. Hence, the functor $\widetilde W$ is also a $\mu$-accessible functor as a colimit of $\mu$-accessible functors.

Since we have chosen $\mu\unrhd \lambda$, we obtain that $\cat K$ is a locally $\mu$-presentable category. Hence, every $X\in \cat K$ is a $\mu$-filtered colimit of $X_{i}\in \cat K_{\mu}$. Therefore $\hat X\cong \colim_{i} \hat X_{i}$, since the fibrant replacement was chosen to be $\lambda$-accessible, which means that it is also $\mu$-accessible. Then, the map
  $$p(\hat X)\colon \widetilde{W}(\hat X)\to  W(\hat X)$$ 
is a $\mu$-filtered colimit of trivial fibrations 
  $$p(\hat X_{i})\colon \widetilde{W}(\hat X_{i})\tilde \fibr W(\hat X_{i})$$ 
with $X_{i}\in K_{\mu}$, since both $W$ and $\widetilde W$ are $\mu$-accessible functors. Therefore, all the maps $p(\hat X)$, $X\in\cat K$ are trivial fibrations.

Given a fibrant object $X\in \cat K$, it is a retract of its fibrant replacement $\hat X$. Therefore, the map $p(X)$ is a retract of $p(\hat X)$ by naturality of $p$, i.e. $p(X)$ is a trivial fibration for all fibrant $X$. We conclude that $p$ is in \cal S.
\end{proof}

We have completed the proof that the fibrant-projective model structure on the category of small functors exists. Now we show that it is equipped with an additional structure of a $\cat V$-model category.

\begin{proposition}\label{V-model-structure}
The fibrant-projective model structure makes $\cat V^\cat K$ into a $\cat V$-model category.
\end{proposition}

\begin{proof}
We will show that for any cofibration $i\colon A\cofib B$ and for any fibration $p\colon X\fibr Y$ in $\cat V^{\cat K}$, the induced map $\hom_\square(i,p)\colon \hom(B, X)\to \hom(A,X)\times_{\hom(A, Y)}\hom(B,Y)$ is a fibration. Moreover, $\hom_\square(i,p)$ is a weak equivalence if either $i$, or $p$ is. 

The retract argument shows that it suffices to prove the statement for cellular cofibrations. We proceed by induction on the construction of the cellular (trivial) cofibration  $i$. Suppose for induction that $B_{0}=A$ and the statement is true for all cardinals smaller than $d$. If $d$ is a successor cardinal, then there is a pushout square  
\[
\xymatrix{
R^{Z}\otimes K
\ar@{^(->}[d]_{R^X\otimes j}
\ar[r]
	& B_{d-1}
	   \ar@{^(->}[d]^{i_d}\\
R^{Z}\otimes L
\ar[r]
	& B_{d}
}
\]
with $j\colon K\cofib L$ in $I_{\cat V}$ (resp., in $J_{\cat V}$) and $Z\in \cat K^{f}$.

We apply $\hom_\square(-,p)$ to the above pushout square, obtaining the following commutative diagram with the left and the right vertical faces being pullback squares.
\[
\xymatrix{
	&	&	&	X(Z)^L%
				\ar	'[d]+<0pt,-10pt>*{\hole}
					'[d]+<0pt,-25pt>*{\hole}
					[dd]
				\ar[rr]
				\ar@{..>}[dr]	&	& 	Y(Z)^L
								\ar[dd]\\
	&	&	&	&	P
					\ar[ur]
					\ar[dl]\\
X^{B_d}
\ar[rr]
\ar[dd]
\ar[uurrr]
\ar@{-->}[dr] &	& Y^{B_d}
				\ar[dd]
				\ar[uurrr]
								&	X(Z)^K
									\ar[rr] &	&	 Y(Z)^K\\
	& Q
	  \ar[ur]
	  \ar[dl]
	  \ar[uurrr]|!{[ur];[dr]}\hole &	&	&\\
X^{B_{d-1}}
\ar[rr]
\ar[uurrr]|!{[uurr];[rr]}\hole &	& Y^{B_{d-1}}
										\ar[uurrr]
}
\]
Let $P=X(Z)^K\times_{Y(Z)^K}Y(Z)^L$ and $Q=X^{B_{d-1}}\times_{Y^{B_{d-1}}}Y^{B_d}$, then also $Q=X^{B_{d-1}}\times_{Y(Z)^K}Y(Z)^L$ as a concatenation of two pullback squares. Applying \cite[Proposition~7.2.14(2)]{Hirschhorn} twice, we conclude, first, that  $Q=X^{B_{d-1}}\times_{X(Z)^K}P$, and next, that $X^{B_d}=Q\times_P X(Z)^L$. 

Then $\hom_\square(i_d,p)\colon \hom(B_{d}, X)\to \hom(B_{d-1},X)\times_{\hom(B_{d-1}, Y)}\hom(B_{d},Y)$ (the dashed map in the front face of the diagram above) is a (trivial) fibration as a base change of the (trivial) fibration 
\[
\hom_\square(R^Z\otimes j,p)\colon \hom(R^Z\otimes L, X)\to \hom(R^Z\otimes K,X)\times_{\hom(R^Z\otimes K, Y)}\hom(R^Z\otimes L,Y)
\]
(the dotted map in the back face of the diagram above). The latter map is a (trivial) fibration, since, by adjunction, it is equal to 
\[
\hom_\square(j,p^{R^Z})\colon \hom(L,X(Z))\to \hom(K,X(Z))\times_{\hom(K, Y(Z))}\hom(L,Y(Z)),
\]
which is a trivial fibration by the analog of SM7(b) in the closed symmetric monoidal model category \cat V.

Now consider the following commutative diagram computing $\hom_\square(i_d\ldots i_2i_1,p)$.
\[
\xymatrix@=12pt{
X^{B_d}
\ar[dr]
\ar[dddd]
\ar[rrrr]	&	&	&	& X^{B_{d-1}}
							\ar[dr]
							\ar[dddd]
							\ar[rrrr]	&	&	&	&X^A
														\ar[dddd]\\
	&	Q
		\ar[dr]
		\ar[urrr]
		\ar[lddd]	&	&	&			&	P
											\ar[dddl]
											\ar[urrr]	&		\\
		&	&	Q'
				\ar[urrr]
				\ar[ddll]	&	&			&	\\
\\
Y^{B_d}
\ar[rrrr]	&	&	&	& Y^{B_{d-1}}
							\ar[rrrr]	&	&	&	& Y^A
}
\]
In this diagram $P=Y^{B_{d-1}}\times_{Y^A}X^A$, $Q=Y^{B_d}\times_{Y^{B_{d-1}}}X^{B_{d-1}}$, and $Q'=Y^{B_d}\times_{Y^A}X^A$.

We need to show that the natural map $\hom_\square(i_d\ldots i_2i_1,p)\colon X^{B_d}\to Q'$ is a (trivial) fibration. But the map $\hom_\square(i_d,p)\colon X^{B_d}\to Q$ is a (trivial) fibration by the previous argument (this is the dashed map in the previous diagram), hence, it is sufficient to show that the induced map $Q\to Q'$ is a (trivial) fibration.

Applying \cite[Proposition~7.2.14(2)]{Hirschhorn} twice, we conclude, first, that $Q'= Y^{B_d}\times_{Y^{B_{d-1}}}P$, and next, that $Q=Q'\times_P X^{B_{d-1}}$.

The natural map $\hom_\square(i_{d-1},p)\colon X^{B_{d-1}}\to P$ is a (trivial) fibration by the inductive assumption, hence $Q\to Q'$ is a (trivial) fibration as a base change of $\hom_\square(i_{d-1},p)$.

Continuing this process we conclude that $\hom_\square(i_{d}\ldots i_1,p)$ is a (trivial) fibration as a transfinite inverse composition of (trivial) fibrations also in the case that $d$ is a limit cardinal. Therefore, $\hom_\square(i,p)$ is a (trivial) fibration.
\end{proof}

\section{Models of spectra}\label{models-for-spectra}

We want to apply Theorem~\ref{fibrant-projective} to a model for the stable homotopy category of spectra. Therefore, we need to demonstrate that there are models that satisfy all assumptions.
The model category of $S$-modules from \cite{EKMM} cannot be used here since its unit for the monoidal structure is not cofibrant.

Symmetric spectra over simplicial sets constructed by Hovey/Shipley/Smith~\cite{HSS} serve as an acceptable model for us. The monoid axiom~\ref{monoid-axiom} is proved in \cite[section 5.4]{HSS}. Strong left properness~\ref{Def. strongly left proper} is not explicitly stated. We prove it now.
\begin{lemma}
The stable model structure on symmetric spectra over simplicial sets is strongly left proper.
\end{lemma}

\begin{proof}
We will use freely the language of Hovey et al. in \cite{HSS} and all the references mentioned here are taken from their paper.

Theorem 5.3.7(3) states that, if $f$ is an $S$-cofibration and $g$ a level cofibration, their pushout product $f\,\square\, g$ is a level cofibration. Because any stable cofibration $i$ is an $S$-cofibration and any symmetric spectrum $A$ is level cofibrant, any map of the form $i\wedge A$ is a level cofibration. Because level cofibrations are stable under cobase change and filtered colimits, all maps in $\mc{D}_{\mc{V}}$ are level cofibrations. By Lemma 5.5.3(1), the stable equivalences are stable under cobase change along level cofibrations.
\end{proof}

Now we turn to Lydakis' simplicial functor model \cite{Lydakis} for $\Sp$. The category \cat V is now given by the pointed simplicial functors from finite pointed simplicial sets $\cat S^{\rm fin}_*$ to pointed simplicial sets $\cat S_*$. The symmetric monoidal product $\otimes$ is given by Day's convolution product \cite{Day-closed}. The monoid axiom for the stable model structure on pointed simplicial functors is proved by Dundas et. al.~\cite[Lemma 6.30]{DRO:enriched} for more general source and target categories. 
\begin{lemma}
Lydakis' stable model structure on pointed simplicial functors is strongly left proper.
\end{lemma}

\begin{proof}
Recall from Definition~\ref{Def. strongly left proper} that $\mathcal{D}_{\cat V}$ is the class of relative cell complexes generated by all morphisms of the form $i\otimes A$, where $i$ is a cofibration and $A$ an object in \cat V. 
We claim that all maps in $\mathcal{D}_{\cat V}$ are objectwise cofibrations. 
Since $\cat S_*$ is left proper it suffices to prove that $i\otimes A$ is an objectwise cofibration for $i$ in a generating set of cofibrations and all objects $A$.

Stable cofibrations coincide with the projective ones by~\cite[Lemma 9.4]{Lydakis}. A generating set for projective cofibrations in \cat V is 
  $$ I_{\cat V}=\{ R^X\wedge (\Lambda^n_k)_+\to R^X\wedge(\Delta^n)_+~|~n\ge k\ge 0~,~n>0~,~X\in\cat S^{\rm fin}_*  \}. $$
For $i\in I_{\cat V}$ the map $i\otimes A$ is isomorphic to
  $$ (R^X\otimes A)\wedge (\Lambda^n_k)_+\to (R^X\otimes A)\wedge(\Delta^n)_+ .$$
By \cite[Lemma 5.13]{Lydakis} or \cite[Cor. 2.8]{DRO:enriched}, using Lydakis' assembly map $F\otimes G\to F\circ G$, this is isomorphic to
  $$ (A\circ R^X)\wedge (\Lambda^n_k)_+\to (A\circ R^X)\wedge(\Delta^n)_+ .$$
After evaluating on an arbitrary finite pointed simplicial set $K$ this map
  $$ A\bigl(\mc{S}_*(K,X)\bigr)\wedge i\co A\bigl(\mc{S}_*(K,X)\bigr)\wedge (\Lambda^n_k)_+\to A\bigl(\mc{S}_*(K,X)\bigr)\wedge(\Delta^n)_+ $$
is clearly a cofibration. We have shown that $\mathcal{D}_{\cat V}$ consists of objectwise cofibrations.

Given a stable weak equivalence, we factor it into a trivial stable cofibration followed by a trivial stable fibration. Every cobase change of the first map remains a trivial stable cofibration. The second map is an objectwise weak equivalence and pushes out along an objectwise cofibration to an objectwise weak equivalence by left properness of $\cat S_*$. 
The composite of both cobase changes is a stable weak equivalence.
\end{proof}

In conclusion, if we take symmetric spectra on simplicial sets or Lydakis' simplicial functors as models for $\Sp$, the fibrant-projective model structure exists on the category $\Sp^{\Sp}$ of small endofunctors and is proper. Here follow some properties of the model structure on $\Sp$ that we will use further down the line.

\begin{remark}\label{rem:further-property}
This property of the fibrant-projective model structure on the category $\Sp^{\Sp}$ is used in Lemma~\ref{compact-dual} below: For any cofibrant object $A$ in $\Sp$ the functor $A\wedge -$ maps stable equivalences to stable equivalences. For symmetric spectra this follows from \cite[5.3.10]{HSS}. For simplicial functors this follows from \cite[Thm.~12.6]{Lydakis}.
\end{remark}

The fibrant-projective model structure on $\Sp^{\Sp}$ is simplicial. This is true for both models by the following reasoning.
\begin{lemma}
Suppose that \cat V is a symmetric closed monoidal model category and that $F\co\mc{S}\to\cat V$ is a strict symmetric monoidal left Quillen functor from the category of pointed simplicial sets. Then every \cat V-category is an $\mc{S}$-category where the simplicial tensor is given by
  $$ V\otimes_{\mc{S}}K=V\otimes_{\cat V}F(K). $$
\end{lemma}

\begin{proof}
The pointed simplicial structure is supplied by \cite[Prop. 6.4.3]{Bor:2}. The verification of compatibility is routine.
\end{proof}

Thus it suffices to exhibit a functor $F\co\mc{S}\to\Sp$ as in the previous lemma. No surprises here; for symmetric spectra this is the symmetric suspension spectrum $K\mapsto\Sigma^{\infty}K$ \cite[p. 163]{HSS}. For the Lydakis model $F$ is given by $K\mapsto\Id\wedge K$ where $\Id$ is the inclusion functor of finite pointed simplicial sets to all pointed simplicial sets and the smash is objectwise. Thus, we have

\begin{remark}\label{rem:simplicial-structure}
For either symmetric spectra or Lydakis' simplicial functors as models for $\Sp$ the fibrant-prjective model structure on $\Sp^\Sp$ is simplicial. 
\end{remark}

\begin{remark}\label{rem:fin-pres-dom-codom}
Since both models for spectra are obtained by localization of either the projective model structure \cite{Lydakis}, or the strict model structure \cite{HSS}, the sets of generating cofibrations have finitely presentable domains and codomains.
\end{remark}

\section{Homotopy functors}\label{section-homotopy}
In this section we assume, like in Section~\ref{fibrant-projective-section}, that $\cat{V}$ is a closed symmetric monoidal combinatorial model category and $\cat K$ is a combinatorial $\cat V$-model category, so that the category of small functors supports the fibrant-projective model structure constructed in Theorem~\ref{fibrant-projective}. 
We assume, in addition to the previous assumptions, that $\cat V$ is a strongly left proper model category, so that the fibrant-projective model structure on the category $\cat V^{\cat K}=\cat F$ of small functors is left proper.
This allows us to localize functors in $\cat F$ turning them into a homotopy functors. 
The whole section is subsumed in Lemma~\ref{approx} which later enters in the proof of Proposition~\ref{E-local}.

\begin{definition}\label{def:homotopy-functor}
By a {\it homotopy functor} in $\cat F$ we mean any functor preserving weak equivalences between fibrant objects. 
\end{definition}

Usually, a homotopy functor is required to preserve all weak equivalences.
If desired, a homotopy functor in our sense here may be turned into a usual homotopy functor by precomposing with a fibrant approximation functor in $\cat K$, while preserving the fibrant-projective homotopy type.

\begin{definition}
Consider the class of maps between cofibrant functors:
\[
\cal H = \{R^{B}\to R^{A} \,|\, A\we B \text{ weak equivalence of fibrant objects in } \cat K\}
\]
We use the standard notions of {\it \mc{H}-local object} and {\it\mc{H}-$($local$)$ equivalence} defined by Hirschhorn \cite[3.1.4]{Hirschhorn}.
\end{definition}

\begin{lemma}
A functor in $\cat V^{\cat K}$ is \mc{H}-local if and only if it is fibrant in the fibrant-projective model structure and a homotopy functor in the sense of Definition~\emph{\ref{def:homotopy-functor}}.
\end{lemma}

\begin{proof}
Left to the reader.
\end{proof}

On $\cat V^{\cat K}$ there exists the projective model structure \cite{Chorny-Dwyer} whose fibrant functors are the objectwise fibrant ones. Obviously, every projectively fibrant functor is fibrant-projectively fibrant.
\begin{proposition}\label{hom-loc}
For every small functor $X\in \cat V^{\cat K}$, there exists an $\cal H$-equivalence $\eta_{X}\colon X\to HX$ such that $HX$ is a homotopy functor with fibrant values on fibrant objects.
\end{proposition}

\begin{proof}
Similarly to the proof of Lemma~\ref{postponed}, let $\mu$ be the maximal cardinal between the accessibility rank of the small (hence, accessible) functor $X$ and the degree of accessibility of the subcategory of weak equivalences in the combinatorial model category $\cat V$; then, it suffices to construct a localization of $X$ with respect to the set $\cal H_{\mu}\subset \cal H$ of maps with $\mu$ accessible domains and codomains. Since $\cat V^{\cat K}$ is left proper, it suffices to apply the small object argument with respect to the following set of maps:
\[
\cal L = \Hor (\cal H'_{\mu}) \cup J_{\mu},
\]
where $J_{\mu}\subset J_{\cat F}$ is the subset of generating trivial cofibrations with $\mu$-accessible domains and codomains, $\cal H'_{\mu}$ is a set of cofibrations obtained from $\cal H_{\mu}$, and $\Hor(-)$ denotes the horns on a set of maps defined in \cite[1.3.2]{Hirschhorn}  
\end{proof}

The following corollary is a standard conclusion from the application of the (generalized) small-object argument, \cite{pro-spaces}.

\begin{corollary}
For every map $f\colon X\to Y$, where $Y$ is a fibrant-projectively fibrant homotopy functor, there exists a map $g\colon HX\to Y$, unique up to homotopy, such that $g\eta_{X}=f$.
\end{corollary}

\begin{remark}
We have constructed, so far, for every small functor $F\in \cat V^{\cat K}$ a map into a homotopy functor $F \to HF$, which is initial, up to homotopy, among the maps into arbitrary homotopy functors. Unlike a similar localization in \cite{BCR} for the projective model structure on $\cal S^{\cal S}$, our current construction is not functorial (since it depends on the accessibility rank of a small functor, which we are localizing), so the corresponding left Bousfield localization of the model category is more involved, \cite[3.2]{ClassHomFun}. We do not use the localized model category in this paper. 
\end{remark}

\begin{definition}
Recall from Defintion~\ref{Kmu} that $\cat K_{\mu}$ denotes the full subcategory of $\cat K$ given by the $\mu$-presentable objects. Recall that $\cat K_{\mu}$ is small, since \cat K is locally presentable. Let $\cat K^{\rm cf}_{\mu}$ be the set of fibrant and cofibrant objects in $\cat K_{\mu}$. We define the following set of maps in $\cat V^{\cat K}$:
  $$ C_{\mu}=\{ R^{A}\otimes K \to R^A\otimes L\,|\, A\in\cat K^{\rm cf}_{\mu}, \vertmap K L \in I_{\cat V} \}.$$
And denote the proper class $C=\cup C_\mu$, where the union is indexed by all ordinals. A functor $X$ in $\cat V^{\cat K}$ is called {\it $C$-cellular} if the map $\emptyset\to X$ is in $C_{\mu}$-cell for some $\mu$.
\end{definition}

\begin{proposition}\label{C-cell-approx}
Let $X$ be a homotopy functor in $\cat V^{\cat K}$. Then there exists a fibrant-projective weak equivalence $X_C\we X$ where $X_C$ is $C$-cellular.
\end{proposition}

\begin{proof}
Let $\mu$ be a regular cardinal such that the small functor $X$ is $\mu$-accessible. The construction of the required cofibrant approximation is the same as in Lemma~\ref{postponed}, except that we will use only the cofibrations in $C_{\mu}$.

The application of the small object argument produces a map $X_C\to X$, such that $X_C(A)\to X(A)$ is a weak equivalence for every fibrant and cofibrant object $A\in \cat K_{\mu}$. But for such $A$, every functor $R^{A}$ is a homotopy functor. Moreover, $X_C$ is also a homotopy functor, as may be proved by cellular induction using the Cube Lemma~\cite[13.5.10]{Hirschhorn}. Since $X$ is also a homotopy functor, the map $X_C\to X$ is a fibrant projective equivalence.
\end{proof}

\begin{lemma}\label{approx}
Every small functor is $\cal H$-equivalent to a $C$-cellular functor.
\end{lemma}

\begin{proof}
For every functor $X$, we construct a homotopy approximation using Proposition~\ref{hom-loc}. We obtain an $\cal H$-equivalence $X\to HX$, such that $HX$ is a homotopy functor. Proposition \ref{C-cell-approx} then allows the construction of a cellular approximation for $\tilde HX\to HX$. We obtain a zig-zag
  $$ X\to HX\leftarrow \tilde HX$$
of $\cal H$-local equivalences.
\end{proof}

\section{The Yoneda embedding as a Quillen equivalence} \label{Q-loc}

The important part of this section is Theorem~\ref{Yoneda Quillen equivalence} where we establish that the the Yoneda adjunction~(\ref{adj-V-K})
\begin{equation}\label{adj-Sp}
   Z\co \Sp^{\Sp}\rightleftarrows \Sp^{\op} : \! Y 
\end{equation}
is a Quillen equivalence. One first notes that the counit $ZY(X)\to X$ is an isomorphism for all spectra $X$. We are done once the unit 
  $$\eta_F\colon F\to YZ(F)$$ 
is a weak equivalence for all small functors $F$. Since this is not the case for the fibrant-projective model structure on $\Sp^{\Sp}$ we perform a localization of it which forces $\eta_F$ to become a local equivalence. 
This localization will be a generalization of the Bousfield-Friedlander technique \cite{BF:gamma} where conditions on a coaugmented functor $Q$ are given such that $Q$ becomes the desired localization functor. 
Our generalization of it deals with the existence of a $Q$-local model structure even in situations where $Q$ is not functorial. This is necessary in categories of small functors, since the factorizations are not functorial -- or at least we do not have functorial constructions of these factorizations. We develop this non-functorial localization in Appendix~\ref{gen-BF}. Here we will apply it by exhibiting a $Q$ that suits our purpose.

Before we proceed let us recall the simplicial mapping cylinder construction. For a map $f\co A\to B$ in a simplicial model category we define $\MC(f)$ as the following combined pushout
\[
\xymatrix{ A \ar[r]^{\imath_0}\ar[d]_f & A\sqcup A \ar[r]^-{i_0\sqcup i_1}\ar[d]^{f\sqcup \Id} & A\otimes\Delta^1 \ar[d]^{f'} \\
           B \ar[r] & B\sqcup A \ar[r] & \MC(f)}
\] 
where $\imath_0$ is the inclusion into the first summand and $i_0$, $i_1$ are the inclusions on the bottom and top of the cylinder. It is a standard argument using the right hand pushout to see that the map 
\begin{equation}\label{def:ell1}
  \ell_1=f'i_1\co A\to\MC(f)
\end{equation} 
is a cofibration as long as $B$ is cofibrant. The universal property of the pushout yields a simplicial equivalence $q\co\MC(f)\to B$ with a section given by the lower horizontal map in the previous diagram such that $f=qi_0$. 

The first idea to take for $Q$ the adjunction~(\ref{adj-Sp}) $\eta\colon F\to YZ(F)$ itself does not work because $YZ$ does not preserve fibrant-projective weak equivalences. However, one can do the following: given $F$, consider first its cofibrant replacement $\tilde F$ and apply the left Quillen functor $Z$, then replace $Z(\tilde F)$ by a fibrant $(Z(\tilde F))^{\wedge}$ and the apply the right Quillen functor $Y$. (We put the standard notation of the fibrant replacement $(\widehat{\phantom{X}})$ on the on the righthand side when the hat becomes awkwardly large; $\widehat Z = Z^{\wedge}$ denotes the composition of $Z$ with the fibrant replacement functor.)
The composition $Y\hat Z$ preserves fibrant-projective weak equivalences between fibrant-projectively cofibrant functors. Finally, this construction has to be equipped with a coaugmentation for arbitrary $F$. This uses the simplicial mapping cylinder as follows:
\[
\xymatrix{
\tilde F\ar@{->>}[dd] ^{\dir{~}}\ar[r]^{\eta_{\tilde F}}\ar@{^(->}[dr]_-{\ell_1}\ar@/^17pt/[rr]^{f}& YZ(\tilde F)\ar[r] & Y\widehat{Z(\tilde F)}\\
& \MC(f) \ar[ur]^-{\dir{~}}_q\ar[d]^{\dir{~}}\\
F\ar[r]^{i}& QF,
}
\]
where $f$ is a composition of the unit $\eta_{\tilde F}$ with an application of $Y$ on the fibrant replacement $Z(\tilde F) \to \widehat{Z(\tilde F)}$ in $\Sp^{\op}$ (cofibrant replacement in $\Sp$), and 
  $$QF= F\sqcup_{\tilde F} \MC(f).$$ 
The codomain $Y\widehat{Z\tilde F}$ is cofibrant in the fibrant-projective model structure on $\Sp^{\Sp}$ by Lemma~\ref{cofibrant representables}. Thus, the map $\ell_1\co\tilde F\to\MC(f)$ is a cofibration and the map $\MC(f)\to Y\widehat{Z(\tilde F)}$ is a weak equivalence. Left properness of the fibrant-projective model structure implies that $\MC(f) \to QF$ is a weak equivalence.

The advantage of using the mapping cylinder instead of the factorization into a cofibration followed by a trivial fibration, guaranteed by the model structure, is that the mapping cylinder construction is functorial. The construction of $QF$ still lacks functoriality, since the cofibrant replacements are not functorial in our model category, but the functoriality of the middle step is essential for the verification of various properties of $QF$ in Proposition~\ref{BF-conditions}.

To summarize, we describe the definition stage by stage.
\begin{definition}\label{Q-constr}
For every $F\in \Sp^{\Sp}$ we define $QF$ together with the coaugmentation map $i_{F}\colon F\to QF$ as follows:
\begin{itemize}
\item Choose a cofibrant approximation of $F$ to obtain $\tilde F$;
\item Factor the composition $f$ of the unit of the adjunction (\ref{adj-Sp}) with the map $Y(Z(\tilde F)\trivcofib \widehat{Z(\tilde F)})$ into a cofibration  followed by a weak equivalence in a functorial way: $\tilde F \cofib \MC(f) \we Y\widehat{Z(\tilde F)}$;
\item Put $QF = F\coprod_{\tilde F} \MC(f)$ with the induced map $i_{F}\colon F\to QF$.
\end{itemize}
We also define $Q$ on maps. Given a map between functors, we need to choose a map on their cofibrant replacements using the lifting axiom of the model structure. It is unique up to simplicial homotopy. The rest of the stages in the definition are functorial. Therefore, once the map of cofibrant replacements is chosen, $Qf$ is defined.
\end{definition}

This definition of $i_{F}\colon F\to QF$ gives rise to a homotopy localization construction as in Definition~\ref{non-func-loc}. It remains to check the conditions (\ref{natural})--(\ref{A6}) of the generalized Bousfield-Friedlander localization given in Theorem~\ref{main}. 

\begin{proposition}\label{hom-idemp-check}
The construction $Q$ from Definition~\emph{\ref{Q-constr}} is homotopy idempotent in the sense that $i_{QF}\colon QF\to QQF$ and $Q(i_F)\colon QF\to QQF$ are weak equivalences for all $F$.
\end{proposition}
\begin{proof}
This is a simple diagram chase relying on Yoneda's lemma: $ZY(X)\cong X$ for all spectra $X$.
The map $i_{QF}$ is constructed as follows:
\[
\xymatrix{
\tilde F\ar@{->>}[dd]^{\dir{~}}\ar@{^(->}[dr]_{a}\ar[rr]^{f_{\tilde F}} & & Y\widehat{Z(\tilde F)}\ar[r]^-{\dir{~}}_-{m} & Y\bigl(ZY(Z\tilde F)^{\wedge}\bigr)^{\wedge} \ar@{=}[r] & Y\bigl((Z\tilde F)^{\wedge}\bigr)^{\wedge}\\
& \MC(f_{\tilde F}) \ar[ur]^{\dir{~}}_c\ar[d]^{\dir{~}} \ar[rr]^{\dir{~}}_k \ar[rd]^{\dir{~}}_{b}&  & Y\bigl(Z(\MC(f_{\tilde F}))\bigr)^{\wedge} \ar[d]^-{Y(Zb)^{\wedge}}_-{\dir{~}} \ar[u]_{Y(Zc)^{\wedge}}^{\dir{~}} & \\
F\ar[r]^{i_{F}}& QF \ar@{^(->}[dr]_{i_{QF}}^{\dir{~}} & \widetilde{QF}\ar@{->>}[l]_{\dir{~}}\ar[r]^{\dir{~}}_{f_{\widetilde{QF}}} \ar@{^(->}[dr]^{\dir{~}}_{l} & Y\bigl(Z\widetilde{QF}\bigr)^{\wedge} & \\
  &  & QQF & \MC(f_{\widetilde{QF}}).\ar[u]^{\dir{~}} \ar[l]_{\dir{~}} & 
}
\]
Here $(-)^{\wedge}$ replaces the hat notation for fibrant replacement.
We first conclude that the upper horizontal map $m=f_{Y(Z(\tilde F))^{\wedge}}$ is a weak equivalence, since this is an application of $Y$ on a fibrant approximation of a fibrant object. Next, we apply the `2-out-of-3' axiom to the lower horizontal maps $k$ and $f_{\widetilde{QF}}$ concluding that they are weak equivalences too. Finally we can see that $i_{QF}$ is a trivial cofibration as a cobase change of the trivial cofibration $l$, which is a weak equivalence by the `2-out-of-3' axiom again.

Now we turn to $Q(i_{F})$ which is depicted on the right in the diagram below.
It suffices to show that the map $\gamma=Y(Z\widetilde{i_F})^{\wedge}$ is a weak equivalence. 
One of the possibilities for choosing a cofibrant approximation to $i_{F}$ is to take the composition $ba$ from the commutative diagram above. 
Consider the following commutative diagram:
\[
\xymatrix@=19pt{
\tilde F \ar@/_30pt/[ddd]_-{\widetilde{i_F}}\ar[dd]_a \ar[rrrr]^{f_{\tilde F}}\ar[dr]^{f_{\tilde F}}& & & &Y\widehat{Z\tilde F}\ar[dl]_{\delta}\ar[dd]_{Y(Za)^{\wedge}}\ar@/^50pt/[ddd]^-{\gamma}\\
 & Y\widehat{Z\tilde F}\ar[rr]_-{f_{Y(Z\tilde F)^{\wedge}}} && Y\bigl(ZY\bigl(Z\tilde F\bigr)^{\wedge}\bigr)^{\wedge}\\
\MC(f_F)\ar[d]^{\dir{~}}_b \ar[ur]^{\dir{~}} \ar[rrrr]^{f_{\MC(f_{\tilde F})}} & & & & Y\bigl(Z(\MC(f_{\tilde F}))\bigr)^{\wedge}\ar[ul]_{\dir{~}} \ar[d]^-{\dir{~}}_-{Y(Zb)^{\wedge}}\\
\widetilde{QF}\ar[rrrr]^{f_{\widetilde{QF}}} & & & & Y(Z\widetilde{QF})^{\wedge}
}
\]
The three lower horizontal maps are weak equivalences, but this is irrelevant to the proof.
The map $\delta=Y\widehat{Zf_F}$ is a weak equivalence since it is weakly equivalent to a second fibrant replacement by Yoneda's lemma as used before: $ZY(X)\cong X$ for every spectrum $X$. Hence, the map $Y\widehat{Za}$ is a weak equivalence by the `2-out-of-3' property. Therefore, the composition $Y\widehat{Zb}\circ Y\widehat{Za}=Y\widehat{Z(ba)}=Y\widehat{Z\widetilde{i_F}}$ is a weak equivalence.
\end{proof}

The following proposition verifies condition \ref{natural}.

\begin{proposition}\label{A2-check}
Let $f\colon F\to G$ be a natural transformation of functors in $\Sp^{\Sp}$; then, $Q_f i_F = i_Gf$, i.e., the following square is commutative.
$$\xymatrix{
F \ar[r]^{i_F}\ar[d]_{f} & QF\ar[d]^{Qf}\\
G \ar[r]_{i_G}         & QG
}$$
\end{proposition}
\begin{proof}
Following the definition of $Qf$, we notice that the only non-functorial stage of the definition is computing the cofibrant replacement of the domain and the codomain of $f$.  But we choose a map $f'\colon \tilde F \to \tilde G$, so that the square
$$\xymatrix{
F \ar[d]_{f} & \tilde F \ar@{->>}[l]_{\dir{~}} \ar[d]^{f}\\
G          & \tilde G \ar@{->>}[l]^{\dir{~}}
}$$
becomes commutative. The remaining steps in the definition are functorial, and hence we end up with the required commutative square.
\end{proof}

Our next goal is to verify that $Q$ satisfies conditions \ref{retract} and \ref{2-out-of-3}. Again, the verification would be immediate if $Q$ were a functorial localization construction. Our approach to this question is to show that $Q$ induces  a functor on the level of homotopy category. Let $\Gamma\colon \Sp^{\Sp} \to {\textup{Ho}}(\Sp^{\Sp})$ be the canonical functor.
\begin{lemma}\label{funct}
The $Q$-construction is a functor up to homotopy: The composition $\Gamma Q\colon \Sp^{\Sp} \to {\textup{Ho}}(\Sp^{\Sp})$ is a functor too.
\end{lemma}
\begin{proof}
For any commutative triangle
\begin{equation}\label{triangle}
\xymatrix{
                & B
                    \ar[dr]^{g}\\
A
\ar[ur]^{f}
\ar[rr]_{h}
               & & C
}
\end{equation}
in $\Sp^{\Sp}$, we have to show that the triangle

\begin{equation}\label{Q-triangle}
\xymatrix{
                & QB
                    \ar[dr]^{Qg}\\
QA
\ar[ur]^{Qf}
\ar[rr]_{Qh}
               & & QC
}
\end{equation}
is commutative up to homotopy, i.e., if we apply on it the functor $\Gamma$, we obtain a commutative triangle in $\Ho(\Sp^{\Sp})$.

We will follow the stages of the construction of triangle (\ref{Q-triangle}) and make sure that at each stage the commutativity is preserved up to homotopy. Recall that the fibrant-projective model structure on $\Sp^\Sp$ is simplicial by Remark~\ref{rem:simplicial-structure}.

The first stage is applying a cofibrant replacement on the vertices  of triangle (\ref{triangle}) obtaining the following triangle with the edges constructed using the lifting axiom.
\begin{equation}\label{cofib-triangle}
\xymatrix{
                & \tilde B
                    \ar[dr]^{\tilde g}\\
\tilde A
\ar[ur]^{\tilde f}
\ar[rr]_{\tilde h}
               & & \tilde C
}
\end{equation}
Triangle (\ref{cofib-triangle}) is commutative up to simplicial homotopy by \cite[Prop.~9.6.1]{Hirschhorn}, since the maps $\tilde h$ and $\tilde g\tilde f$ are the lifts in the commutative square
\[
\xymatrix{
\emptyset
\ar@{^{(}->}[d]
\ar[r] 	& \tilde C
		    \ar@{->>}[d]^{\dir{~}}\\
\tilde A
\ar[r]		& C.
}
\]
The next stage in the construction of $Q$ is the application of simplicial functors $Z$,$Y$ and the functorial cofibrant replacement in spectra (fibrant replacement in $\Sp^{\op}$) in between. Simplicial functors preserve simplicial homotopies of maps. Cofibrant replacement in any simplicial model category
allows for the lift of simplicial homotopy: if $J\in \sS$ is a generalized interval, the simplicial homotopy of $Z\tilde h$ and $Z\tilde f Z\tilde g$ is a map $H\colon Z\tilde A \to Z\tilde C^{J}$, such that $\ev_{0}H = Z\tilde h$ and $\ev_{1}H = Z\tilde gZ\tilde f$, and hence $H$ can be lifted to a simplicial homotopy $\tilde H\colon  \widetilde{Z\tilde A} \to \widetilde{Z\tilde C}^J$
\[
\xymatrix{
\emptyset
\ar@{^{(}->}[d]
\ar[rr] 	&		& \widetilde{Z\tilde C}^{J}
		    \ar@{->>}[d]^{\dir{~}}\\
\widetilde{Z\tilde A}
\ar@{->>}[r]^{\dir{~}}
\ar@{..>}[urr]^{\tilde H}	&
			Z\tilde A
			\ar[r]_{H}	& Z\tilde C^{J},
}
\]
so that each of the simplicially homotopic maps $\ev_{0}\tilde H$ and $\ev_{1}\tilde H$ is a lift to the cofibrant replacements of the maps $Z\tilde h$ and $Z\tilde f Z\tilde g$, respectively. On the other hand, the maps $\ev_{0}\tilde H$ and $\ev_{1}\tilde H$ are simplicially homotopic to the functorially induced maps of cofibrant replacements in $\Sp$, i.e., the maps $\widetilde{Z\tilde h}$ and $\widetilde{Z\tilde g}\widetilde{Z\tilde f}$ are simplicially homotopic by transitivity of the simplicial homotopy relation.

So far, we have obtained two triangles commutative up to simplicial homotopy with a natural map between them:
\[
\xymatrix{
		&	& Y\widehat{Z\tilde B}
				\ar[dr]^{Y\widehat{Z\tilde g}}\\
		& Y\widehat{Z\tilde A}
			\ar[ur]^{Y\widehat{Z\tilde f}}
			\ar[rr]_{Y\widehat{Z\tilde h}}
               &  & Y\widehat{Z\tilde C}\\
                & \tilde B
                    \ar[dr]^{\tilde g}
                    \ar@{..>}[uur]\\
\tilde A
\ar[ur]^{\tilde f}
\ar[rr]_{\tilde h}
\ar@{..>}[uur]
               & & \tilde C
               		\ar@{..>}[uur]
}
\]
The completion of the localization construction involves factoring the dotted maps into cofibrations followed by a weak equivalence and then applying the cobase change. Both operations are natural and change only the commuting triangle in the homotopy category up to a natural isomorphism, preserving the commutativity.
\end{proof}

\begin{proposition}\label{A3-A4-check}
The localization construction $Q$ satisfies conditions \emph{\ref{retract}} and \emph{\ref{2-out-of-3}}.
\end{proposition}
\begin{proof}
It follows immediately from Lemma~\ref{funct}.
\end{proof}

The following property is reminiscent of functoriality and verifies \ref{new_funct}.

\begin{proposition}\label{prop-comm-cube}
For every commutative square of small functors
\begin{equation}\label{comm-square}
\xymatrix{
A \ar[r]^{h}\ar[d]_{f} & X\ar[d]^{g}\\
B \ar[r]_{k}         & Y
}
\end{equation}
there exists a commutative cube
\begin{equation}\label{comm-cube}
\xymatrix@=15pt{
          & Q'A\ar[rr]^{h'}\ar'[d][dd]^{f'}& &Q'X\ar[dd]^{g'}\\
A \ar[rr]^>>>>>>>>{h}\ar[dd]_{f}\ar[ur]& & X\ar[dd]^<<<<<<{g}\ar[ur]\\
          & Q'B\ar'[r]_>>>>{k'}[rr]& &Q'Y\\
B \ar[rr]_{k}\ar[ur]        & & Y\ar[ur]
}
\end{equation}
for some choice of $Q'A\simeq QA$, $Q'B \simeq QB$, $Q'X \simeq QX$, $Q'Y \simeq QY$. Moreover, every edge of the cube connecting the front face with the back face factors through the corresponding $Q$ construction, i.e., $A \to QA\we Q'A$, $B\to QB \we Q'B$, $X\to QX\we QX'$, and $Y\to QY\we Q'Y$.
\end{proposition}
\begin{proof}
Given a commutative square (\ref{comm-square}), we will go through the stages of Definition~\ref{Q-constr} and make sure that the commutativity of the diagram can be resolved at each stage of the construction, so that at the end we obtain the commutative cube (\ref{comm-cube}).

The first stage is to take cofibrant replacements of all the vertexes of the commutative square (\ref{comm-square}). Since $\Sp^\Sp$ is a simplicial model category by Remark~\ref{rem:simplicial-structure}, the lifts existing by Quillen's MC5 are unique up to simplicial homotopy. In other words, for any choice of cofibrant replacements of the entries in our commutative square (\ref{comm-square}), the maps between them may be constructed using MC5 and the obtained cube will be commutative, except for the back face, which will commute up to simplicial homotopy
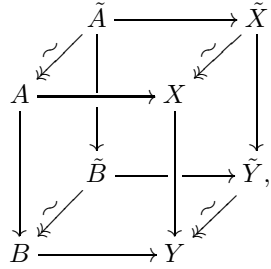
\[
\xymatrix@=15pt{
          & \tilde A\ar[rr]\ar'[d][dd]\ar@{->>}[dl]_{\dir{~}}& & \tilde X\ar[dd]\ar@{->>}[dl]_{\dir{~}}\\
A \ar[rr]\ar[dd]& & X\ar[dd]\\
          & \tilde B\ar'[r][rr]\ar@{->>}[dl]_{\dir{~}}& &\tilde Y\ar@{->>}[dl]_{\dir{~}},\\
B \ar[rr]        & & Y
}
\]
since the two possible maps $\tilde A\to \tilde Y$ form a lift in the commutative square
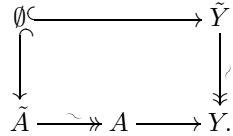
\[
\xymatrix{
\emptyset \ar@{^(->}[rr] \ar@{^(->}[d]& & \tilde Y\ar@{->>}[d]^{\dir{~}}\\
\tilde A \ar@{->>}[r]^{\dir{~}} & A\ar[r] & Y.
}
\]

In other words, there exists a cylinder object $\tilde A\wedge I$ such that the diagram
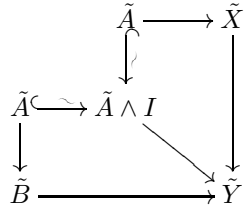
\[
\xymatrix@=18pt{
  &\tilde A \ar@{^(->}[d]^-{\dir{~}} \ar[r]&  \tilde X\ar[dd]\\
\tilde A \ar@{^(->}[r]^-{\dir{~}}\ar[d]& \tilde A\wedge I \ar[dr]\\
\tilde B \ar[rr] & & \tilde Y
}
\]
is commutative.

Thus we can alternate the original choice of the cofibrant replacements so that the whole cube will become commutative. Replace the back face of the original cube by the following (dotted) commutative square:
\begin{equation}\label{corrected-back-face}
\xymatrix{
  &\tilde A \ar@{^(->}[d]^{\dir{~}}_{i_{1}} \ar[r]&  \tilde X\ar@{^(->}[d]^{\dir{~}}\\
\tilde A \ar@{^(->}[r]^{\dir{~}}_{i_{0}}\ar[d]& \tilde A\wedge I \ar[dr] \ar[d]\ar[r]& \tilde X\coprod_{\tilde A}\tilde A\wedge I \ar[d]\\
\tilde B \ar@{^(->}[r]^<<<<{\dir{~}} & \tilde B\coprod_{\tilde A}\tilde A\wedge I \ar[r]& \tilde Y
\save "2,2"."3,3"*[F.]\frm{}
\restore
} 
\end{equation}
The possibility of incorporating this commutative square into the original commutative cube is given by the map $\tilde A\wedge I\to \tilde A$ left inverse to both $i_{0}$ and $i_{1}$. It also ensures that all the old vertices  of the cube are retracts on the new ones.

We obtain the commutative cube
\begin{equation}\label{cofibr-cube}
\xymatrix@=15pt{
          & \tilde A'\ar[rr]\ar'[d][dd]\ar[dl]_{\dir{~}}& & \tilde X'\ar[dd]\ar[dl]_{\dir{~}}\\
A \ar[rr]\ar[dd]& & X\ar[dd]\\
          & \tilde B'\ar'[r][rr]\ar[dl]_{\dir{~}}& &\tilde Y\ar@{->>}[dl]_{\dir{~}},\\
B \ar[rr]        & & Y
}
\end{equation}
where $\tilde A' = \tilde A\wedge I$, $\tilde B'=\tilde B\coprod_{\tilde A}\tilde A\wedge I$, and $\tilde X'=\tilde X\coprod_{\tilde A}\tilde A\wedge I$.

Note that all the new vertexes of the commutative cube are related to the old ones by trivial cofibrations. The rest of the stages of Definition~\ref{Q-constr} are functorial, and hence they produce the required commutative cube (\ref{comm-cube}), and turn the trivial cofibrations between the old and the new vertices into weak equivalences, factorizating the maps between the vertices in the front and in the back face of this cube through the corresponding $Q$-constructions.
\end{proof}

The following explicit characterization of $Q$-equivalences facilitates the verification of the rest of the conditions required from $Q$-construction by Theorem~\ref{main}.
\begin{proposition}\label{alt-classification}
A map $f\colon F\to G$ of small functors is a $Q$-equivalence from Definition~\emph{\ref{non-func-loc}} if and only if for any cofibrant replacement $\tilde f\colon \tilde F\to \tilde G$ of $f$, the induced map $\hom(\tilde f, \Id_{\Sp})\colon \hom(\tilde G, \Id_{\Sp})\to\hom(\tilde F,\Id_{\Sp})$ is a weak equivalence of spectra.
\end{proposition}
\begin{proof}
Readily follows from the construction of $Q$.
\end{proof}

Everything we have said so far may be said about the category of small functors from spaces to spaces. The next proposition uses the properties of the stable model category in an essential way. We do not know if its analog is true in the category of small functors from spaces to spaces. 

Note that the (fibrant-)projective model structure on $\Sp^\Sp$ is stable since weak equvalences, homotopy pushouts and homotopy pullbacks are objectwise (resp., in fibrant objects), and spectra form a stable model category.

\begin{proposition}\label{A6-check}
A base change of a $Q$-equivalence along  a $Q$-fibration is a $Q$-equivalence.
\end{proposition}
\begin{proof}
We would like to apply Proposition~\ref{alt-classification} in order to check if the map $f\colon A\to B$ in the pullback square
\[
\xymatrix{
A\ar[d]_{f}\ar[r]^{g} & B\ar[d]^{Q\text{-eq.}}\\
C\ar@{->>}[r]_{Q\text{-fib.}}^{f'} & D
}
\]
is a $Q$-equivalence. The $Q$-fibration $f'$ is a fibrant-projective fibration by Theorem~\ref{main} and the square is a homotopy pullback.
Let $\tilde f\colon A'\to C'$ be  a cofibrant approximation of $f$ with a weak equivalence of maps $(r_{A},r_{B})\colon \tilde f\to f$. Factor the composition $gr_{A}$ into a cofibration $g'\colon A'\hookrightarrow B'$ followed by a weak equivalence $q\colon B' \tilde\to B$. Let $D'=B'\coprod_{A'} C'$. Then the induced map $D'\to D$ is a weak equivalence, since in the fibrant-projective model category on $\Sp^{\Sp}$ homotopy pullbacks are also homotopy pushouts.  Therefore, the following pushout square of cofibrant objects is levelwise weakly equivalent to the original square:
\[
\xymatrix{
A'\ar[d]_{\tilde f}\ar@{^{(}->}[r]^{g'} & B'\ar[d]^{Q\text{-eq.}}\\
C'\ar[r] & D'.
}
\]
The map $B'\to D'$ is a $Q$-equivalence by the `2-out-of-3' property for $Q$-equivalences \ref{2-out-of-3} verified in Proposition~\ref{A3-A4-check}.

Applying $Z$, we obtain the homotopy pullback square of spectra
\[
\xymatrix{
\hom(D', \Id)\ar[d]_{\dir{~}}\ar[r] & \hom(C', \Id)\ar[d]^{Z(\tilde f)}\\
\hom(B', \Id)\ar@{->>}[r] & \hom(A', \Id),
}
\]
which is, in turn, a homotopy pushout of spectra. Hence $Z(\tilde f)$ is a weak equivalence. By Proposition~\ref{alt-classification}, $\tilde f$ is a $Q$-equivalence, and hence $f$ is a $Q$-equivalence by the `2-out-of-3' property.
\end{proof}

Now we are ready to verify the conditions \ref{natural}--\ref{A6} of the Theorem~ \ref{main}.

\begin{theorem}{\label{BF-conditions}}
The homotopy localization construction $Q$ satisfies the conditions \emph{\ref{natural}--\ref{A6}}.  Therefore, by Theorem~\emph{{\ref{main}}}, there exists the $Q$-local model structure on the category of small functors from spectra to spectra.
\end{theorem}
\begin{proof}
Construction $Q$ is a homotopy localization construction, since $Q$ is homotopy idempotent by Proposition~\ref{hom-idemp-check} and preserves weak equivalences. Condition \ref{natural} was verified in Proposition~\ref{A2-check}. \ref{retract} and \ref{2-out-of-3} were shown in Proposition~\ref{A3-A4-check}. Condition \ref{new_funct} was proved  in Proposition~\ref{prop-comm-cube}, Condition~\ref{A6} in Proposition~\ref{A6-check}.

By the generalized Bousfield-Friedlander Theorem \ref{main}, there exists the $Q$-local model structure on $\Sp^{\Sp}$ denoted by $\Sp^{\Sp}_{Q}$
\end{proof}

\begin{lemma}\label{repr-Q-local}
For all cofibrant $X\in \Sp$, the representable functor $Y(X)=R^X$ is $Q$-local.
\end{lemma}
\begin{proof}
For every cofibrant $X\in \Sp$ the represented functor $Y(X)=R^X$ is fibrant in the fibrant-projective model structure. It remains to show that $Q(Y(X))\simeq Y(X)$.

The $Q$-construction begins with the cofibrant replacement of $Y(X)$. Consider the fibrant replacement $X\trivcofib \hat X$ of $X$ in $\Sp$. Then $R^{\hat X}=Y(\hat X)\trivfibr Y(X)=R^{X}$ is a cofibrant replacement in the fibrant-projective model structure by Lemma~\ref{right-lifting-properties}. The unit of the adjunction (\ref{adj-Sp}) is the identity in our case, and $Y(\widetilde{\hat X} \trivfibr \hat X)$ is a weak equivalence $f\colon R^{\hat X}\tilde\to R^{\widetilde{\hat X}}$ since $X$ and hence $\widetilde{\hat X}$ and $\hat X$ are cofibrant.

The factorization of $f$ into a cofibration followed by a weak equivalence produces a cofibration, which is also a weak equivalence by the `2-out-of-3' property. Hence, its cobase change $Y(X)\to Q(Y(X))$ is a weak equivalence again; in fact, a trivial cofibration.
\end{proof}

The main result of this section is the following
\begin{theorem}\label{Yoneda Quillen equivalence}
The adjunction \emph{(\ref{adj-Sp})} becomes a Quillen equivalence after we localize the left-hand side with respect to $Q$.
\end{theorem}
\begin{proof}
First, we need to show that the adjunction (\ref{adj-Sp}) is still a Quillen adjunction after we localize the left hand side with respect to $Q$. It suffices to check, by Dugger's lemma \cite[Prop. 8.5.4]{Hirschhorn}, that the right adjoint $Y$ preserves fibrations of fibrant objects and all trivial fibrations. By Lemma~\ref{repr-Q-local}, $Y(X)$ is $Q$-local for all cofibrant $X\in \Sp$ or, equivalently, fibrant $X\in \Sp^{\op}$. Hence $Y$ applied on a fibration of fibrant objects produces a fibration of $Q$-local objects, i.e., a $Q$-fibration by Lemma~\ref{A.8}(\ref{s3}). Trivial fibrations do not change under the $Q$-localization by Lemma~\ref{A.8}(\ref{s2}), and hence are preserved by $Y$ as in Proposition~\ref{Quillen-adjunction}.

Given a cofibrant $A\in (\Sp^{\Sp})^{Q}$ and a fibrant $X\in \Sp^{\op}$, we need to show that a map $f\colon Z(A)\to X$ is a weak equivalence if and only if the adjoint map $g\colon A\to YX$ is a $Q$-equivalence.

Suppose that $f$ is a weak equivalence and consider the fibrant replacement $j\colon Z(A)\tilde\cofib \widehat{Z(A)}$. Then there exists a lift $\hat f\colon \widehat{Z(A)}\to X$ satisfying $f=\hat f j$, since $X$ is fibrant. By `2-out-of-3,' $\hat f$ is a weak equivalence. The adjoint map $g$ may be factored as the unit $\eta_{A}\colon A\to YZA$ composed with $Yf=Y\hat f\circ Yj$. But $Y\hat f$ is a weak equivalence, since $\hat f$ is a weak equivalence of fibrant objects and $Y$ is a right Quillen functor. The composition $Yj\circ \eta_{A}$ is a $Q$-equivalence by definition of $Q$ and Proposition~\ref{hom-idemp-check}. Therefore $g$ is a $Q$-equivalence.

Conversely, suppose $g$ is a $Q$-equivalence. Let $p\colon X \trivfibr \hat X$ be a fibrant replacement of $X$ in $\Sp$, in other words a cofibrant replacement in $\Sp^{\op}$. 
Then $Yp\colon Y\hat X \trivfibr YX$ is a cofibrant replacement in $(\Sp^{\Sp})^{Q}$ in the fibrant-projective model structure by Lemma~\ref{cofibrant representables}. Then, there exists a lift $\hat g\colon A\to Y\hat X$ satisfying $g=Yp\circ \hat g$. Moreover, $\hat g$ is a weak equivalence by the `2-out-of-3' axiom. The adjoint map may be factored as $Zg=ZYp\circ Z\hat g$ composed with the counit of the adjunction $\varepsilon_{X}\colon ZYX\to X$. By Yoneda's lemma $ZYp=p$ and $\varepsilon_{X}=\Id_{X}$.  Hence, $f=p\circ Z\hat g$, but $\hat g$ is a weak equivalence between cofibrant objects, and therefore $Z\hat g$ is a weak equivalence, since $Z$ is a left Quillen functor.
\end{proof}

This model is much more complicated than the dual of any other model of spectra that we know, but it has a nice advantage. 

\begin{corollary}\label{cor:aleph-0-small}
Every object in $\Sp^\Sp$ is $Q$-local weakly equivalent to a representable functor whose representing spectrum is cofibrant. In particular, all objects are $Q$-local weakly equivalent to an $\aleph_{0}$-small object.
\end{corollary}

\begin{proof}
For a functor $F\co\Sp\to\Sp$ set $\widehat{Z(\tilde F)}=X\in\Sp$. Then we have $Q$-local equivalences
  $$ F\simeq Q(F)\simeq Y\widehat{Z(\tilde F)}=R^X. $$
By the Yoneda lemma mapping out of representable functors commutes with all colimits in $\Sp^\Sp$ because they are computed objectwise.
\end{proof}

The closely related Theorem~\ref{dual-Brown} is one of our main results. Here is another illustration of the advantage.

\begin{example}
Consider a not necessarily compact spectrum $A$ and its associated homology functor $A\wedge -\co\Sp\to\Sp$. What is the best approximation of this functor by a representable functor? If $A$ happens to a be the Spanier-Whitehead dual of some spectrum $B$, i.e. $A\simeq \hom(B,\hat S)=D B$, then there is a natural map $A\wedge - \to \hom(B,-)=R^B$, which is adjoint to the evaluation map $A\wedge B \to \hat S$ smashed with the identity functor.

Nevertheless, this is not the best approximation of our functor. Computing the fibrant replacement in the localized model structure, we obtain a map $A\wedge - \to \hom(D A,-)$, which turns out to be a better approximation, since there is a map $\hom(D A,-)\to\hom(B,-)$ induced by the natural morphism $B\to DD B$. The usual lifting property in the model category allows one to construct a factorization of any natural transformation $A\wedge - \to \hom(C,-)$, where $C$ is cofibrant, through a functor weakly equivalent to $\hom(D A,-)$.
\end{example}

Unfortunately, the $Q$-local model structure on $\Sp^\Sp$ is not as nice as we could hope for. It is not \emph{class-cofibrantly generated} \cite{pro-spaces}. Here is the reason: class-cofibrantly generated model categories have the property that fibrant objects are closed under $\mu$-filtered colimits for any $\mu$ bigger than the presentability rank of the domains and the codomains of the generating trivial cofibrations. In the $Q$-local model category, the fibrant objects are weakly equivalent to the functors represented by cofibrant objects. However, representable functors are not closed under filtered colimits of any cardinality. See \cite{Chorny-Brownrep} for more details and examples of non-class-cofibrantly generated model categories.
This drawback makes it very difficult to perform localizations or cellularizations in our new model category.

\section{Enriched representability in the dual category of spectra}\label{sec:rep}

In this section, we use our model of the opposite category of spectra to prove an enriched version of the Brown representability theorem for that category. This theorem classifies representable functors up to homotopy in terms of their commutation with certain homotopy limits.

In Section~\ref{Q-loc} we have established a Quillen equivalence between the $Q$-local model structure on $\Sp^\Sp$ and the category $\Sp^\op$. The latter is equivalent to the full subcategory of representable functors. 
We are going to prove that fibrant functors in the $Q$-local model structure are precisely those functors that commute up to weak equivalences with the required homotopy colimits. This establishes our representability theorem.

In more detail, starting from the fibrant-projective model structure on $\Sp^{\Sp}$ we will prove that the $Q$-localization constructed in Section~\ref{Q-loc} is precisely the localization that ensures that the local objects are those functors that take homotopy pullbacks to homotopy pullbacks and commute with products up to homotopy. In other words, we have to show that $Q$-localization is the localization with respect to the class $\cal E=\cal F_1\cup\cal F_2$ of maps where
\[
\cal F_{1}=\left\{\left.
\hocolim\left(
\vcenter{\xymatrix@=10pt{
R^{D}\ar[r]\ar[d]&R^{B}\\
R^{C}
}}
\right)\longrightarrow R^{A}
\right|
\vcenter{
\xymatrix@=10pt{
A\ar[r]\ar[d]&B\ar[d]\\
C\ar[r]&D
}
}\text{ -- \begin{minipage}{1.4in}homotopy pullback\\ of fibrant objects in $\Sp$ \end{minipage}}
\right\}
\]
and
\[
\cal F_{2}=\left\{\left.
\coprod R^{X_{i}}\longrightarrow R^{\prod X_{i}}
\right|
\text{for fibrant } X_{i}\in \Sp,\,i\in \cal I
\right\}
\]

\begin{proposition}\label{maps-in-E-are-Q-equiv}
All maps in $\cal E$ are $Q$-local weak equivalences. 
\end{proposition}

\begin{proof}
We need to show that the application of $Q$ on every map in $\cal E$ results in a weak equivalence. This is readily verified by going through the stages of Definition~\ref{Q-constr} and applying Yoneda's lemma and the fact that mapping out of a homotopy colimit results in a homotopy limit. 
\end{proof}

We are now going to show the converse: it suffices to invert all the maps in $\cal E=\cal F_{1}\cup \cal F_{2}$ in order to obtain the $Q$-local model structure. Since we know, by Corollary~\ref{cor:aleph-0-small}, that the $Q$-local objects are precisely the fibrant functors fibrant-projectively equivalent to the functors represented in cofibrant spectra, it suffices to show that every functor is $\cal E$-equivalent to a representable functor so that we can conclude that the $Q$-fibrant objects coincide with the $\cal E$-local objects, and  hence $Q$-equivalences coincide with $\cal E$-equivalences.

\begin{lemma}\label{compact-dual}
Let $A$ and $X$ be cofibrant spectra and suppose that $A$ is compact. Then there is fibrant-projective equivalence  
  $$A\wedge R^{X}\simeq R^{DA\wedge X}$$
natural in $A$ and $X$, where $DA$ is a cofibrant representative of the Spanier-Whitehead dual of $A$.
\end{lemma}

\begin{proof}
Let us first establish a special case of this equivalence. Suppose $X=\mathbb{S}$, the sphere spectrum, so that $R^{\mathbb S}=\Id_{\Sp}$. There is then a natural map, $(A\wedge -)\to\hom(DA,-)$, corresponding by adjunction to the evaluation map $A\wedge DA\to \mathbb{S}$ smashed with the identity map of identity functors.

The functor $A\wedge -$ is a homotopy functor according to Remarks~\ref{rem:further-property}. On the other hand, $\hom(DA,-)$ preserves weak equivalences of fibrant spectra, so that if we compose it with the fibrant replacement functor, it becomes a homotopy functor. If we show that the composition $A\wedge \Id\to \hom(DA,\Id)\to \hom(DA,\hat{\Id})$ is an objectwise equivalence of functors, we will conclude that the initial map is a fibrant-projective equivalence of functors.

In order to show that the composed map of functors is a levelwise weak equivalence, consider the derived natural transformation of derived functors on the homotopy category of spectra $A\wedge \Id\to [DA,\Id]$, where the total derived functors exist since the original functors preserve weak equivalences. For functors defined on the homotopy category of spectra, this map is an isomorphism of functors if and only if $A$ is strongly dualizable \cite{Dold-Puppe}.
Further on, $A$ is a strongly dualizable spectrum iff $A$ is compact \cite[3.1]{Dold-Puppe}.

It remains only to apply these two equivalent functors on the representable functor $R^{X}$ in order to obtain the required equivalence: $A\wedge R^{X}\we \hom(DA,R^{X})=R^{DA\wedge X}$.
\end{proof}

\begin{proposition}\label{E-local}
Every small functor $W\in \Sp^{\Sp}$ is $\cal E$-equivalent to a functor represented in a fibrant and cofibrant spectrum.
\end{proposition}

\begin{proof}
Recall from Section~\ref{section-homotopy} that $\cal H =\{R^{D}\to R^{C} \,|\, C\we D \text{ weak equivalence in } \Sp \}$ and note that $\cal H\subset \cal E$, since for every weak equivalence $C\we D$ in \Sp, there is a homotopy pullback square
\[
\xymatrix{
C
\ar[r]^{\dir{~}}
\ar[d]_{\dir{~}}  &  D
				\ar@{=}[d]\\
D
\ar@{=}[r]  &   D
}.
\]
Hence the map $\hocolim(R^{D}\overset = \leftarrow R^{D} \overset = \rightarrow R^{D})=R^{D}\to R^{C}$ is in $\cal F_{1}\subset \cal E$. Therefore, every $\cal H$-equivalence is also an $\cal E$-equivalence. By Lemma~\ref{approx}, the small functor $W$ is $\cal H$-equivalent, and hence $\cal E$-equivalent, to an $I$-cellular complex $W'$ that may be decomposed into a colimit indexed by a cardinal $\lambda$:
\[
W'=\colim_{a<\lambda}(W_{0}\to\cdots\to W_{a}\to W_{a+1}\to\cdots),
\]
where $W_{0}=0=R^{0}$ is the functor associating the zero spectrum to every entry, and $W_{a+1}$ is obtained from $W_{a}$ by attaching an $I$-cell:
\[
\xymatrix{
R^{\hat X}\wedge A \ar[r]\ar@{^(->}[d]& W_{a}\ar[d]\\
R^{\hat X}\wedge B \ar[r] & W_{a+1},
}
\]
where $i\colon A\hookrightarrow B$ is a generating cofibration of spectra, i.e., $A$ and $B$ are compact spectra by \ref{rem:fin-pres-dom-codom}, and $\hat X$ is a fibrant and \emph{cofibrant} spectrum.

By \cite[Lemma~3.3]{Chorny-Brownrep} we may replace the above decomposition of $W'$ with  a countable sequence $W'=\colim_{a<\omega}W'_{a}$, such that at every stage a coproduct of a set of cells is attached instead of just one cell, as in the inner square of the following commutative diagram.

\begin{equation}\label{comm-diag}
\xymatrix@=18pt{
R^{\prod_{i,\hat X}\hom(A,\hat X)}
\ar[ddd]
\ar[rrr] 						 &	&	& R^{Y_{a}}
										\ar[ddd]^{f_{a}}\\
	& \coprod_{i, \hat X} R^{\hat X}\wedge A
	\ar[r]
	\ar@{^(->}[d]
	\ar[ul]					 & W'_{a}
								\ar@{^(->}[d]
								\ar[ur]\\
	& \coprod_{i, \hat X} R^{\hat X}\wedge B
	\ar[r]
	\ar[dl]					  & W'_{a+1}
								\ar@{-->}[dr] \\
R^{\prod_{i,\hat X}\hom(B,\hat X)}
\ar[rrr]						&	&	&R^{Y_{a+1}}\\
}
\end{equation}
Assume for induction that there is an $\cal E$-equivalence $W'_a \to R^{Y_{a}}$ with $Y_a$ fibrant and cofibrant. The diagonal arrows pointing left in the diagram above are the units of adjunction (\ref{adj-Sp}). The upper horizontal arrow exists by the universal property of the unit of adjunction.

We can conclude that $W'_{a+1}$ is $\cal E$-equivalent to a representable functor $R^{Y_{a+1}}$, where $Y_{a+1}$ is computed as follows.

Since $\hat X$ is cofibrant and $A,B$ are compact, Lemma~\ref{compact-dual} implies that there are weak equivalences $R^{\hat X}\wedge A\we R^{DA\wedge\hat X}$ and $R^{\hat X}\wedge B\we R^{DB\wedge\hat X}$, and hence the left vertices of the inner commutative square in the commutative diagram (\ref{comm-diag}) are $\cal E$-equivalent to the representable functors $R^{\prod(DA\wedge \hat X)}$ and $R^{\prod(DB\wedge \hat X)}$. Notice that $\hom(A,\hat X)\simeq DA\wedge\hat X$ and $\hom(B,\hat X)\simeq DB\wedge\hat X$ by the proof of Lemma~\ref{compact-dual}, since we can substitute $DA$ and $DB$ instead of $A$ and $B$, respectively. Hence, all the solid diagonal arrows in (\ref{comm-diag}) are $\cal E$-equivalences.

Set $Y_{a+1}=Y_a\times_{\prod_{i,\hat X}\hom(A,\hat X)}\prod_{i,\hat X}\hom(B,\hat X)$. Then, this is also a homotopy pullback, since the commutative square
\[
\xymatrix{
\prod_{i,\hat X}\hom(A,\hat X) & Y_{a}\ar[l]\\
\prod_{i,\hat X}\hom(B,\hat X)\ar@{->>}[u] &Y_{a+1}\ar[l]\ar@{->>}[u]
}
\]
is a homotopy pullback of spectra. Therefore, for every $\cal E$-local small functor $U$, the mapping of the commutative diagram (\ref{comm-diag}) into $U$ induces weak equivalences on all diagonal arrows in (\ref{comm-diag}). Hence, the dashed arrow is also an $\cal E$-equivalence.

Now we need to show that $W' = \colim_{a<\omega} W'_{a}$ is $\cal E$-equivalent to a representable functor.

So far, we have constructed the following countable commutative ladder
\[
\xymatrix{
W'_{0} \ar[r] \ar[d]           & W'_{1} \ar[r] \ar[d]   & \cdots \ar[r]  & W'_{a} \ar[r] \ar[d] & \cdots \\
R^{Y_0}  \ar[r]^{f_0}   & R^{Y_1} \ar[r]^{f_1}  & \cdots \ar[r]^{f_{a-1}}  & R^{Y_a} \ar[r]^{f_a}      & \cdots,
}
\]
where vertical arrows are $\cal E$-equivalences.

Both $\omega$-indexed colimits in the above ladder are homotopy colimits in the fibrant-projective model structure, since the generating cofibrations have finitely presentable domains and codomains. Therefore, trivial fibrations are preserved under filtered colimits. Hence, the induced map 
  $$\colim_{a<\omega} W'_{a}\simeq \hocolim_{a<\omega} W'_{a}\to \hocolim_{a<\omega} R^{Y_{a}}\simeq \colim_{a<\omega} R^{Y_{a}}$$ 
is an $\cal E$-equivalence, which can be verified by mapping to an arbitrary $\cal E$-local object.

It remains to show that $\colim_{i<\omega} R^{Y_{i}}$ is $\cal E$-equivalent to a representable functor.
The $\omega$-indexed colimit may be represented as a pushout square:
\[
\xymatrix{
\left(\coprod_{a<\omega} R^{Y_a}\right)\sqcup \left(\coprod_{a<\omega} R^{Y_a}\right)\ar[rr]^-{1\sqcup (\sqcup f_{a})}\ar[d]_{\nabla}&&\coprod_{a<\omega} R^{Y_a}\\
\coprod_{a<\omega} R^{Y_a}, && 
}
\]
where $\nabla$ is the codiagonal and the horizontal map is combined of identity morphism and the sum of bonding maps. We observe that the double mapping cylinder of the diagram above is weakly equivalent to the telescope construction applied to the sequence $\{R^{Y_i}\}$, and therefore the homotopy pushout is weakly equivalent to the sequential homotopy colimit. In the following natural morphism of pushout   diagrams the vertical maps are $\cal E$-equivalences from $\cal F_{2}$
\[
\xymatrix{
\coprod_{a<\omega} R^{Y_a} \ar[d]& (\coprod_{a<\omega} R^{Y_a})\sqcup (\coprod_{a<\omega} R^{Y_a})\ar[rr]^-{1\sqcup (\sqcup f_{a})}\ar[l]_>>>>>{\nabla} \ar[d] &&\coprod_{a<\omega} R^{Y_a} \ar[d]\\
R^{\prod_{a<\omega} Y_a}  &  R^{(\prod_{a<\omega} Y_a)\times (\prod_{a<\omega} Y_a)} \ar[l]\ar[rr] && R^{\prod_{a<\omega} Y_a}.
}
\]
Hence the induced map of homotopy pushouts is also an $\cal E$-equivalence.

However, the homotopy colimit of the lower row is $\cal E$-equivalent to the representable functor $R^P$, where
\[
P={\holim\left(\left(\prod_{a<\omega} Y_a \right) \longrightarrow \left(\prod_{a<\omega} Y_a \right)\times \left(\prod_{a<\omega} Y_a \right) \longleftarrow \left(\prod_{a<\omega} Y_a \right)\right)}.
\]
By consideration dual to the homotopy telescope construction, it can be argued that $P\simeq \holim_{a<\omega}Y_{a}\simeq \lim_{a<\omega}Y_{a}$, since all the bonding maps are fibrations.

The representing object $P={\lim_{a<\lambda}Y_a}$ is a fibrant spectrum, but not necessarily cofibrant.  Consider the homotopy pulback square
\[
\xymatrix{
P\ar@{=}[r]&P\ar@{=}[d]\\
\tilde P\ar@{->>}[u]^{\dir{~}}\ar@{->>}[r]^{\dir{~}}&P,
}.
\]
Then, the map $\hocolim(\xymatrix{R^{P}&R^{P}\ar@{=}[r]\ar@{=}[l]&R^{P}})=R^{P}\to R^{\tilde P}$ is an $\cal E$-equivalence.

Finally, $W$ is $\cal E$-equivalent to a functor represented in a fibrant and cofibrant object $R^{\tilde P}$.
\end{proof}

\begin{theorem}\label{dual-Brown}
Let $F\colon \Sp\to \Sp$ be a small functor. Assume that $F$ takes homotopy pullbacks to homotopy pullbacks and also preserves arbitrary products up to homotopy. Then there exists a cofibrant spectrum $Y$ such that $F\simeq R^{Y}$ in the fibrant-projective model structure.
\end{theorem}
\begin{proof}
Let $F$ be a functor satisfying the conditions of the theorem. Consider its fibrant replacement in the fibrant-projective model structure $F\trivcofib \hat F$. Then $\hat F$ is an $\cal E$-local functor. Proposition~\ref{E-local} and the local Whitehead theorem show that every $\cal E$-local functor is (fibrant-projective) equivalent to a functor represented in a fibrant and cofibrant object.
\end{proof}

\begin{corollary}
The $Q$-localization of the category of small functors is precisely the localization with respect to the class $\cal E$ of maps.
\end{corollary}
\begin{proof}
All the maps in $\cal E$ are $Q$-equivalences by Proposition~\ref{maps-in-E-are-Q-equiv}. Moreover, by Theorem~\ref{dual-Brown} the $\cal E$-local objects are precisely the $Q$-local objects, hence the class of $\cal E$ coincides with the class of $Q$-equivalences.
\end{proof}

\appendix

\section{A non-functorial Bousfield-Friedlander localization}\label{gen-BF}
In this section, we generalize the Bousfield-Friedlander localization machinery, originally devised in \cite[Appendix A]{BF:gamma} and improved on by Bousfield in \cite[Section 9]{Bou:telescopic}, so that it will apply to the localization constructions, which are not necessarily functorial. Let us assume that the model category \cat C is both left and right proper in this appendix, so that we can use the elementary properties of homotopy pushouts and homotopy pullbacks freely.

\begin{definition}\label{non-func-loc}
A (non-functorial) \emph{homotopy localization construction} $Q$ in a model category \cat C is an assignment of a map $\eta_{X}\colon X\to QX$ for every $X\in \cat C$ and of a map $Qf\colon QX\to QY$ for every map $f\colon X\to Y$ in $\cat C$, such that for all $X\in \cat C$ the maps $\eta_{QX}, Q\eta_{X}\colon QX\to QQX$ are weak equivalences and $Qf$ is a weak equivalence for all weak equivalences $f$ in $\cat C$. A map $f\colon X\to Y$ in \cat C is a $Q$\emph{-equivalence} if $Qf\colon QX\to QY$ is a weak equivalence, a $Q$\emph{-cofibration} if $f$ is a cofibration, and a $Q$\emph{-fibration} if the filler exists in each commutative diagram
\[
\xymatrix{
A\ar[r]\ar@{^{(}->}[d]_{i}& X\ar[d]^{f}\\
B\ar[r] & Y
}
\]
where $i$ is a $Q$-cofibration and a $Q$-equivalence.
\end{definition}

We consider the following conditions on a homotopy localization construction $Q$ in the category $\cat C$.

\begin{cond} \label{natural}
For all maps $f\colon X\to Y$ in $\cat C$, $\eta_{Y}f = Qf\eta_{X}$, i.e., the square
\[
\xymatrix{
X\ar[r]^{\eta_{X}}\ar[d]_{f}& QX\ar[d]^{Qf}\\
Y\ar[r]_{\eta_{Y}} & QY
}
\]
is commutative.
\end{cond}

\begin{cond} \label{retract}
Any retract of a $Q$-equivalence is a $Q$-equivalence
\end{cond}

\begin{cond} \label{2-out-of-3}
$Q$-equivalences satisfy the ``2-out-of-3'' property.
\end{cond}

\begin{cond}\label{new_funct}
For all commutative squares
\[
\xymatrix{
X_{1} \ar[d]_{f_{12}}\ar[r]^{f_{13}}&X_{3}\ar[d]^{f_{34}}\\
X_{2}\ar[r]_{f_{24}}&X_{4}
}
\]
in \cat C there exists a commutative cube
   $$\xy
   \xymatrix"*"@=13pt{ Q'X_{1} \ar[rr] \ar'[d] [dd] & & Q'X_{3} \ar[dd] \\
                                                & &                         \\
              Q'X_{2}  \ar'[r] [rr] & & Q'X_{4}  }
   \POS(-10.5,-9)
   \xymatrix@=17pt{ X_{1} \ar["*"] \ar[rr]\ar[dd] & & X_{3} \ar[dd] \ar["*"]\\
                                  & &                \\
                    X_{2} \ar[rr] \ar["*"] & &  X_{4} \ar["*"]  }
   \endxy
   $$
such that for all $1\leq a\leq 4$ the map $X_a\to Q'X_{a}$ factors as $\eta_{X_a}\colon X_a \to QX_a$ composed with a weak equivalence $QX_a \we Q'X_{a}$, and for all $1\leq a<b\leq 4$ the map  $Q'f_{ab}\colon Q'X_{a}\to Q'X_{b}$ is a weak equivalence if and only if  $Qf_{ab}$ is.
\end{cond}

The following strengthening of condition \ref{new_funct} is not used in the proof of the localization theorem, but it is required for the ``only if" part in the classification of $Q$-fibrations. 

\begin{condstarred}\label{more_funct}
In addition to the conditions of \ref{new_funct} we require that every morphism of maps $f_{ab}\to Q'f_{ab}$ in the commutative cube of \ref{new_funct} factors through $\eta_{f_{ab}}\colon f_{ab}\to Qf_{ab}$, which exists by \ref{natural}.
\end{condstarred}

Note that in \ref{more_funct} we do not require that the square of maps $Qf_{ab}$ commutes.

The following classical condition (cf. \cite[A.6]{BF:gamma}) is necessary for the localization theorem
\begin{cond} \label{A6}
If in the pullback square
\[
\xymatrix{
W\ar[r]\ar[d]_{g}& X\ar[d]^{f}\\
Z\ar@{->>}[r]_{h} & Y
}
\]
$h$ is a $Q$-fibration, $f$ is a $Q$-equivalence, then $g$ is a $Q$-equivalence.
\end{cond}

The following proposition was proved in \cite[A.1]{BF:gamma}.
\begin{proposition}\label{A.1}
Let \cat C be a proper model category and let $f\colon X\to Y$ in \cat C. For each factorization $[f]=vu$ in $\Ho(\cat C)$ there is a factorization $f=ji$ in \cat C such that $i$ is a cofibration, $j$ is a fibration, and the factorization $[f]=[j][i]$ is equivalent to $[f]=vu$ in $\Ho(\cat C)$ (i.e., there exists an isomorphism $w$ in $\Ho(\cat C)$ such that $wu=[i]$ and $[j]w=v$.)
\end{proposition}

The main goal of this appendix is to prove the following
\begin{theorem}\label{main}
Given a localization construction $Q$ in a model category \cat C satisfying  \emph{\ref{natural}, \ref{retract}, \ref{2-out-of-3}, \ref{new_funct}}, and \emph{\ref{A6}}, the category \cat C equipped with $Q$-equivalences as weak equivalences, $Q$-cofibrations as cofibrations and $Q$-fibrations as fibrations is a right proper model category denoted by $\cat C^{Q}$. Moreover,  a map $f\colon X\to Y$ in \cat C is a $Q$-fibration if
and only if 
$f$ is a fibration and
\[
\xymatrix{
X \ar[r]^{\eta_{X}} \ar[d]_{f} & QX \ar[d]^{Qf}\\
Y \ar[r]_{\eta_{Y}} & QY
}
\]
is a homotopy pullback square in \cat C. The ``only if" direction of this classification of $Q$-fibrations depends on the additional assumption \ref{more_funct}.
\end{theorem}

\begin{lemma}\label{A.8}
Given a localization construction $Q$ in a model category \cat C satisfying \emph{\ref{natural}, \ref{retract}, \ref{2-out-of-3}}, and \emph{\ref{new_funct}}, then
\begin{enumerate}
\item \label{s1}  $\cat C^{Q}$ satisfies {\rm CM1--CM4} and the ``cofibration, trivial fibration'' part of {\rm CM5};
\item \label{s2} A map $f\colon X\to Y$ in \cat C is a trivial fibration in $\cat C^{Q}$ iff $f$ is a trivial fibration in $\cat C$;
\item \label{s3} If $f\colon X\to Y$ is a fibration in \cat C and both $\eta_{X}\colon X\to QX$ and $\eta_{Y}\colon Y\to QY$ are weak equivalences, then $f$ is a $Q$-fibration.
\end{enumerate}
\end{lemma}
\begin{proof}
We follow the plan of the original proof \cite[A.8]{BF:gamma} specifying the changes necessary for our generalization.

We start with statement (\ref{s2}), since it is used for the proof of (\ref{s1}). The ``if'' direction of (\ref{s2}) follows from definitions and ``only if'' follows by first factoring $f$ as $f=ji$, with $i$ a cofibration and $j$ a trivial fibration, and then noting that $f$ is a retract of $j$ by a lifting argument using the fact that $i$ is a $Q$-equivalence by (\ref{2-out-of-3}). For (\ref{s3}), it suffices to show that the filler exists in each commutative square
\[
\xymatrix{
A \ar[r] \ar@{^{(}->}[d]_{i} &X\ar[d]^{f}\\
B\ar[r]\ar@{-->}[ur]&Y
}
\]
with $i$ a trivial cofibration in $\cat C^{Q}$.

Consider the Reedy model structure on the category $\cat C^\text{Pairs}$. Then the commutative square above may be viewed as a map $i\to f$. Applying \ref{new_funct}, we obtain the commutative diagram
\[
\xymatrix{
          & Q'A\ar[rr]\ar'[d][dd]& &Q'X\ar[dd]\\
A \ar[rr]\ar[dd]_{i}\ar[ur]& & X\ar[dd]^<<<<<{f}\ar[ur]\\
          & Q'B\ar'[r][rr]& &Q'Y\\
B \ar[rr]\ar[ur]        & & Y\ar[ur]
}
\]
equipped with the factorizations on the right side face:
\[
\xymatrix{
A\ar[d]_i\ar[r] & Q'A\ar[r]\ar[d]_{Q'i} & Q'X \ar[d]_{Q'f} & QX \ar[l]_\simeq  & X\ar[l]^{\eta_{X}}_{\simeq}\ar[d]^{f}\\
B\ar[r] &  Q'B\ar[r] & Q'Y & QY \ar[l]_\simeq & Y.\ar[l]^{\eta_{Y}}_\simeq
}
\]
Therefore, the original map of maps $i\to f$ factors in $\Ho(\cat C^\text{Pairs})$ through $[Q'i]$, which is an isomorphism since $[Qi]$ is. Applying \ref{A.1}, we obtain a commutative diagram
\[
\xymatrix{
A\ar[d]_i\ar[r] & V\ar[d]_{h}\ar[r] & X \ar[d]^{f}\\
B\ar[r] & W\ar[r]  & Y,
}
\]
where $h$ is isomorphic to $Q'i$ in $\Ho(\cat C^\text{Pairs})$. Then $h$ is a weak equivalence, and therefore we apply {\rm CM5} to $h$ and use {\rm CM4} to obtain the desired filler.
\end{proof}

\begin{lemma}\label{if-dir}
If $f\colon X\to Y$ is  a fibration in $\cat C$ and
\[
\xymatrix{
X\ar[r]^{\eta_X}\ar[d]_{f} & QX\ar[d]^{Qf}\\
Y\ar[r]_{\eta_Y} & QY}
\]
is a homotopy pullback square, then $f$ is a $Q$-fibration.
\end{lemma}
\begin{proof}
Let $i\colon A\to B$ be a trivial cofibration in $\cat C^{Q}$. We need to construct a lift in any commutative square $i\to f$.
\[
\xymatrix{
A
\ar@{^(->}[dd]_i
\ar[rr]
  &   &  X
	\ar[dd]_f
	\ar[rr]
	\ar[dr]^w_{\dir{~}}  &   & QX
					\ar[dd]
					\ar@{^(->}[dr]^u_{\dir{~}}\\
  &   &     & P
		\ar[rr]|\hole
		\ar@{->>}[dl]_{v'}		 &   &  S
							\ar@{->>}[dl]^v\\
B
\ar[rr]
  &   &  Y
	\ar[rr]  &   & QY
}
\]

Consider first the composed map $i\to Qf$ and factor $Qf$ as $Qf = vu$ where $u\colon QX\to S$ is a trivial cofibration in \cat C and $v\colon S\to QY$ is a fibration. Let $P=S\times_{QY} Y$ be the pullback; then, the induced map $w\colon X\to P$ is a weak equivalence by assumption.

By Lemma \ref{A.8}(\ref{s3}) $v$ is a $Q$-fibration, hence there exists the filler $B\to S$, and hence $B\to P$ by the universal property of the pullback. If we factor $w$ as $w=kl$ where $l\colon X\to T$ is a cofibration and $k\colon T\to P$ is a fibration and both are weak equivalences, then there exists a lift $B\to T$  in the commutative square $i\to k$, since $i$ is a cofibration and $k$ is a trivial fibration in \cat C.

\[
\xymatrix{
X
\ar@{=}[rr]
\ar@{^(->}[d]_l^{\dir{~}}   &     &   X
						\ar@{->>}[d]^f\\
T
\ar@{->>}[r]_k^{\dir{~}} 
\ar@{-->}[urr]			  &  P
					\ar@{->>}[r]_{v'} &   Y
}
\]
 Next, $X$ is a retract of $T$ over $Y$, since $f$ is a fibration and there exists a lift in the above commutative square. We construct the required filler by composing the lift $B\to T$ with the retracting map $T\to X$. Commutativity of the bottom triangle in the square $i\to f$ with the lift $B\to X$ constructed follows from the commutativity of the above diagram.

Different argument: since the above retraction is over $Y$, then the map $f$ is a retract of the composition $v'k$, which is a $Q$-fibration, since $v'$ is a base change of a $Q$-fibration $v$ and $k$ is a trivial fibration in \cat C, hence $f$ is also a $Q$-fibration. 
\end{proof}

The following definition was suggested by A.~K.~Bousfield in a private correspondence together with the ``only if'' direction in the previous lemma.

\begin{definition}
A map $h$ in \cat C is called \emph{$Q$-compatible} if the commutative square $h\to Qh$ is a homotopy pullback square.
\end{definition}

\begin{remark}
$Q$-compatible maps are closed under composition and retracts due to corresponding properties of homotopy pullback squares, \cite[A2]{BF:gamma}

\end{remark}

\begin{proof}[Proof of Theorem~\emph{\ref{main}}]
It remains to factor a map $f\colon X\to Y$ in \cat C as $f=ji$, where $i$ is a $Q$-cofibration and $Q$-equivalence and $j$ is a $Q$-fibration. The proof is the same as  in \cite[A.10]{BF:gamma}.
\[
\xymatrix{
X\ar[ddd]_f \ar[rrr]^{\eta_{X}} \ar[ddr]_{u'} \ar@{^{(}->}[dr]^{i}
   & & & QX\ar[ddd]^{Qf} \ar@{_{(}->}[ddl]^u_{\dir{~}}\\
   & R \ar@{->>}[d]^{k}_<<<{\dir{~}}\\
   & T \ar@{->>}[dl]_{v'} \ar[r]& S \ar@{->>}[dr]_{v}\\
Y\ar[rrr]_{\eta_{Y}}
   & & & QY
}
\]
First factor $Qf$ as $Qf=vu$, where $u\colon QX\trivcofib S$ is a trivial cofibration and $v$ is a fibration in \cat C. Let $T= S\times_{QY} Y$. The natural map $u' \colon X\to T$ together with $v'$ factor $f$ as $f=v'u'$. The fibration  $v$ is a $Q$-fibration by \ref{A.8}(\ref{s3}) and $v'$ is a base change of $v$, hence also a $Q$-fibration. Moreover, the base change of the $Q$-equivalence $\eta_{Y}$ along a $Q$-fibration $v$ is a $Q$-equivalence by \ref{A6}, and hence $u'$ is a $Q$-equivalence by \ref{2-out-of-3}.

Factor $u'$ as $u'=ki$, where $i$ is  a cofibration and $k$ is a trivial cofibration, and hence, a $Q$-fibration by \ref{A.8}(\ref{s2}). Then the factorization $f=(v'k)i$ has the desired properties, since $v'k$ is a composition of two $Q$-fibrations and $i$ is a $Q$-equivalence by \ref{2-out-of-3}.

The ``if'' direction of the classification of fibrations is Lemma~\ref{if-dir}. 

The ``only if'' direction follows from the observation that $v'$ is a $Q$-compatible map. 
By \ref{new_funct} there is a commutative cube 
   $$\xy
   \xymatrix"*"@=13pt{ Q'T \ar[rr] \ar'[d] [dd] & & Q'S \ar[dd] \\
                                                & &                         \\
              Q'Y  \ar'[r] [rr] & & Q'QY  }
   \POS(-10.5,-9)
   \xymatrix@=17pt{ T \ar["*"] \ar[rr]\ar[dd]_{v'} & & S \ar[dd]_>>>>>v \ar["*"]\\
                                  & &                \\
                    Y \ar[rr]_{\eta_Y} \ar["*"] & &  QY \ar["*"]  }
   \endxy
   $$
with the front face being a homotopy pullback by construction and the right face also a homotopy pullback, since both slanted arrows are weak equivalences. Hence, the combined square is also a homotopy pullback. The back face of the cube is also a homotopy pullback since its horizontal maps are weak equivalences: $Q'\eta_Y\colon Q'Y\to Q'QY$ is a weak equivalence by \ref{new_funct} since $Q\eta_Y$ is a weak equivalence, and the map $Q'T\to Q'S$ is a weak equivalence by combining \ref{new_funct} and \ref{A6}. Since the cube is commutative the left face is also a homotopy pullback. 
Assumption \ref{new_funct} does not guarantee that it is possible to factor the slanted morphisms of the left face through $Qv'$. 
This is possible by the additional assumption, \ref{more_funct}. With it we conclude that $v'$ is $Q$-compatible. 

Next we see that $k$ is $Q$-compatible since it is a weak equivalence. Thus $v'k$ is $Q$-compatible as a composition of $Q$-compatible maps.

Given that the map $f$ is a $Q$-fibration, we need to show that $f$ is also $Q$-compatible. The map $i$ is a cofibration and a $Q$-equivalence, and hence $X$ is a retract of $R$ over $Y$. Therefore $f$ is a retract of $v'k$, i.e.,  a $Q$-compatible map, as desired.
\end{proof}

\begin{remark}
The ``if'' direction of the classification of fibrations is proved in Lemma~\ref{if-dir} and  does not rely on condition \ref{A6}, while the ``only if'' direction, proven in Theorem~\ref{main}, relies on \ref{A6} and on an additional condition \ref{more_funct}.
\end{remark}

\bibliographystyle{abbrv}
\bibliography{Xbib}

\end{document}